\documentclass[reqno]{amsart}
\usepackage{amsmath,amsfonts,amssymb,pstricks}
\numberwithin{equation}{section}
\title[]{The Lip-lip condition on metric measure spaces}
\author{Jasun Gong}
\address{Jasun Gong \hfill\break\indent
Institute of Mathematics \hfill\break\indent
Aalto University \hfill\break\indent 
P.O. Box 11100  \hfill\break\indent
FI-00076 Aalto \hfill\break\indent
Finland
}
\email{jasun.gong@aalto.fi}
\date{14 August 2012}
\thanks{This research was supported by a grant from the Academy of Finland.}
\theoremstyle{plain}
\newtheorem{thm}{Theorem}[section]

\newtheorem{lemma}[thm]{Lemma}
\newtheorem{prop}[thm]{Proposition}
\newtheorem{cor}[thm]{Corollary}
\newtheorem{ques}[thm]{Question}
\theoremstyle{definition}
\newtheorem{defn}[thm]{Definition}

\newtheorem{rmk}[thm]{Remark}

\newtheorem{claim}[thm]{Claim}

\renewcommand{\d}{\delta}
\newcommand{\e}{\epsilon}
\renewcommand{\H}{\mathcal{L}}
\newcommand{\id}{\operatorname{id}}
\newcommand{\Lip}{\operatornamewithlimits{Lip}}
\newcommand{\lip}{\operatornamewithlimits{lip}}
\newcommand{\N}{\mathbb{N}}
\newcommand{\R}{\mathbb{R}}
\newcommand{\Tan}{{\rm Tan}}
\newcommand{\To}{\operatornamewithlimits{\longrightarrow}}
\newcommand{\U}{\Upsilon}
\newcommand{\wslim}{\operatornamewithlimits{{\rm w}^*{\rm lim}}}
\newcommand{\wsto}{\stackrel{*}{\rightharpoonup}}
\newcommand{\x}{\textbf{x}}

\begin{document}
\maketitle

\begin{abstract}
On complete metric spaces that support doubling measures, we show that the validity of a Rademacher theorem for Lipschitz functions can be characterised by Keith's ``Lip-lip'' condition.  Roughly speaking, this means that at almost every point, the infinitesmal behavior of every Lipschitz function is essentially independent of the scales used in the blow-up at that point.  Moreover, the doubling property can be further weakened to a local hypothesis on the measure; we also present results in this direction.

Our techniques of proof are new and may be of independent interest.  They include an explicit use of coordinate charts for measurable differentiable structures, as well as a blow-up procedure on Euclidean spaces that preserves Weaver derivations.
\end{abstract}


\section{Introduction} \label{sect_intro}

\subsection{Rademacher-type theorems on metric spaces}

A well-known theorem of Rademacher \cite{Rademacher} states that Lipschitz functions on $\R^n$ are almost everywhere differentiable with respect to Lebesgue measure.  In recent years, however, there has been much interest in differentiability properties for Lipschitz maps on general metric spaces. We focus here on the case where the source is a metric space, while the target remains Euclidean.

The study of generalised differentiability has deep connections to other parts of mathematics.  Consider, for instance, the problem of characterising metric spaces that allow bi-Lipschitz embeddings into a Euclidean space, which has been met with partial success by means of Rademacher-type theorems for such spaces.  This was first observed by Pansu \cite{Pansu} for Carnot groups, by Semmes \cite{Semmes:2} for certain classes of metric topological manifolds, and then by Cheeger \cite{Cheeger} for a large class of metric spaces without any a priori smoothness.
Moreover, the later work of Cheeger and Kleiner \cite{Cheeger:Kleiner:2, Cheeger:Kleiner:3} shows that such differentiability and non-embedding theorems also hold true for Lipschitz mappings with Banach space targets, which in turn lead to new counter-examples in theoretical computer science \cite{Goemans}, \cite{Lee:Naor}.  

It is therefore a topic of independent interest to study differentiability problems in their own right.  Similar to these embedding problems, one may inquire as to which metric spaces satisfy a Rademacher-type theorem with respect to some Radon measure, in which case the space is said to support a {\em measurable differentiable structure}.  For a precise formulation of this property, see Definition \ref{defn_MDS}.

As indicated before, Cheeger \cite{Cheeger} proved the existence of (non-degenerate) measurable differentiable structures for a large class of metric spaces, specifically those that support doubling measures and a weak Poincar\'e inequality in terms of upper gradients (in the sense of Heinonen and Koskela \cite{Heinonen:Koskela}).  
This was later extended by Keith \cite[Thm 2.3.1]{Keith}, where the Poincar\'e inequality is replaced by one of its implications, called the {\em Lip-lip condition}. 
Roughly speaking, it asserts that at almost every point, the infinitesmal behavior of every Lipschitz function is essentially independent of the scales used in the blow-up at that point.

\begin{thm}[Cheeger]
Let $(X,d)$ be a metric space and let $\mu$ be a doubling measure on $X$.  If $X$ supports a weak $p$-Poincar\'e inequality, for some $p \in [1,\infty)$, then it supports a measurable differentiable structure.
\end{thm}

\begin{thm}[Keith] \label{thm_keith}
Let $(X,d)$ be a locally compact metric space and let $\mu$ be a doubling measure on $X$.
If $(X,d,\mu)$ satisfies a Lip-lip condition with constant $M \geq 1$, that is, for all Lipschitz functions $f : X \to \R$ the inequality
\begin{equation} \label{eq_Liplip}
\Lip[f](x) \;\leq\; M \, \lip[f](x)
\end{equation}
holds $\mu$-a.e.\ $x \in X$,
then it supports a measurable differentiable structure.
\end{thm}

(See \S\ref{subsect_lipconst} for definitions of $\Lip[f](x)$ and $\lip[f](x)$, the upper and lower pointwise Lipschitz constants of $f$ at $x$, respectively.)

For doubling measures, our main result is essentially a converse to Theorem \ref{thm_keith}.  Up to a measurable partition on a metric space, the Lip-lip condition is actually {\em necessary} for measurable differentiable structures to exist on that space.  This also settles a previous question by the author \cite[Ques 1.11]{Gong_diffstruct}.

\begin{thm} \label{thm_converse}
Let $(X,d)$ be a complete metric space with a doubling measure $\mu$.  The following conditions are equivalent:
\begin{enumerate}
\item[(\ref{thm_converse}.A)] $(X,d,\mu)$ supports an $N$-dimensional measurable differentiable structure, for some $N \in \N$;
\vspace{.025in}
\item[(\ref{thm_converse}.B)] There is a collection of measurable subsets $\{Z_n\}_{n=1}^\infty$ of $X$ and a sequence $(M_n)_{n=1}^\infty$ in $[1,\infty)$ with 
$
\mu(X \setminus \bigcup_{n=1}^\infty Z_n) = 0
$
and each $(Z_n,d,\mu)$ satisfies a Lip-lip condition with constant $M_n$.
\end{enumerate}
\end{thm}



As a related phenomenon, Cheeger \cite[Cor 6.36]{Cheeger} has shown that for metric spaces equipped with doubling measures, the validity of a (weak) $p$-Poincar\'e inequality for some $p \in [1,\infty)$ implies a Lip-lip condition with constant $M=1$.  This motivates the following question, as suggested to the author by J.\ Jaramillo.

\begin{ques}
Are there examples of metric measure spaces that satisfy a Lip-lip condition with a constant $M$ strictly greater than $1$?  More concretely, are there examples of measures on $\R^N$ whose supports satisfy a Lip-lip condition with a constant $M$ strictly greater than $1$?
\end{ques}

For $N \leq 2$, a weaker result holds:\ there is a universal constant for the Lip-lip condition.  In fact, for low-dimensional measurable differentiable structures on metric spaces, the (full) converse to Keith's theorem holds:

\begin{cor} \label{cor_converse}
Let $\mu$ be a doubling measure on a complete metric space $(X,d)$.  If $(X,d,\mu)$ supports an $N$-dimensional measurable differentiable structure, for $N \leq 2$, then $X$ satisfies a Lip-lip condition with constant $M=\sqrt{N+1}$.
\end{cor}

This result relies crucially on the fact that pushforwards of the underlying measure into $\R^2$ must be absolutely continuous to Lebesgue measure \cite[Thm 1.2]{Gong_rigidity}.  For $N=1$, this is now standard; see, for example, \cite[p.\ 15]{Ambrosio:Kirchheim}.

\subsection{The use of local coordinates}

Measurable differentiable structures naturally extend the notion of smooth structures on manifolds.  Unlike the latter case, however, Definition \ref{defn_MDS} is rarely handled with explicit coordinate charts.

Existence proofs for such structures on general metric spaces, as first shown by Cheeger \cite{Cheeger}, are often analytic (and non-constructive) in nature.  Specifically they involve dimension bounds for classes of ``quasi-linear'' or generalised harmonic functions on weak tangents of the space.  For Riemannian manifolds with non-negative Ricci curvature, the same phenomena was previously observed by Yau \cite{Yau}, Colding and Minicozzi \cite{Colding:Minicozzi:2,Colding:Minicozzi:1}, Li \cite{Li}, and others.  For subsequent analogues in the metric space setting, see \cite{Keith}, \cite{Kleiner}, \cite{Kleiner:Mackay}, and the recent work \cite{Schioppa}.

In contrast, the proof of Theorem \ref{thm_converse} handles measurable differentiable structures by way of ``passing to local coordinates.''  To obtain Lip-lip conditions from such structures, one pushes forward the relevant data to charts, proves the theorem for Euclidean spaces, and then pulls back the results to the underlying metric space.

The novelty here is that {\em injectivity of coordinate maps is not necessary} to the argument.  It suffices instead to choose chart coordinates appropriate to the metric space and in some cases, to ``lift'' them in order to better fit the geometry.  For the case of doubling measures, coordinates can be chosen as distance functions; this was first observed by Keith \cite{Keith_distcoord} for the case of spaces supporting Poincar\'e inequalities and by Schioppa \cite{Schioppa} for the general case.

\subsection{Differentiability, porosity, and tangents}

One crucial tool in proving Theorem \ref{thm_converse} is a new characterisation of measurable differentiable structures on metric spaces with doubling measures \cite[Thm 1.6]{Gong_diffstruct}.  As formulated in Theorem \ref{thm_MDSchar}, such structures are equivalent to the existence of generalised differential operators --- more precisely, the {\em (metric) derivations} of Weaver \cite{Weaver} --- that satisfy a local-to-global inequality.

So by passing to local coordinates, we show that Lip-lip conditions on Euclidean spaces are roughly opposite to {\em porosity} conditions for the support $S$ of the (pushforward) measure:\ that is, at every point $a \in S$ and every scale there are holes near $a$, lying in $\R^n \setminus S$, and of comparable diameter as the given scale.
The previous characterisation of measurable differentiable structures, now treated as directional differentiability, exploits this porosity by means of ``blow-ups'' at measure density points.  We note that the connection between differentiability and porosity in Euclidean spaces has been studied by Preiss and Zaj\'i\v{c}ek \cite{Preiss:Zajicek:2,Preiss:Zajicek}. 
More recently, Bate and Speight \cite{Bate:Speight} showed that measures on spaces supporting measurable differentiable structures (or {\em Lipschitz differentiability spaces}, in their terminology) cannot be porous; see also Lemma \ref{lemma_MDSptdoubling}.

To run the blow-up procedures mentioned above, we require the notion of a {\em tangent measure}  from geometric measure theory \cite{Marstrand,Preiss,Mattila}, as well as adaptations of the techniques from measurable differentiable structures for them.  In particular, we introduce the notion of {\em tangent derivations}, whose dimension (or rank) as a module is preserved under blowups.  

\subsection{Stronger characterisations of differentiability}

Very recently, the author has learned about a new result by Bate \cite[Thm 8.10]{Bate} which characterises metric spaces supporting a measurable differentiable structure with respect to a Radon measure.  It is important to note that the result {\em does not require any additional hypotheses on the underlying measure}.  In particular, this generalises Theorem \ref{thm_converse} and his methods are independent of ours.

As a difference in terminology, in \cite[Defn 10.3]{Bate} the Lip-lip condition is defined in terms of a countable Borel (measurable) decomposition of $X$, instead of over the entire space $X$.  To keep the discussion self-contained, we formulate his result below in the latter sense.

\begin{thm}[Bate] \label{thm_bate}
Let $(X,d)$ be a metric space and let $\mu$ be a Radon measure on $X$.  Then $(X,d,\mu)$ has a nondegenerate measurable differentiable structure if and only if both of the following conditions hold:
\begin{enumerate}
\item[(\ref{thm_bate}.A)] 
The measure $\mu$ is pointwise doubling, in that $\mu$-almost every $x_0 \in X$ satisfies
$$
\limsup_{r \to 0} \frac{\mu(B(x_0,r))}{\mu(B(x_0,\frac{r}{2}))} \;<\; \infty.
$$
\item[(\ref{thm_bate}.B)] 
There exist a sequence $\{\d_i\}_{i=1}^\infty$ in $\R^+$ and $\mu$-measurable subsets $\{X_i\}_{i=1}^\infty$ of $X$ with 
$\mu(X \setminus \bigcup_{i=1}^\infty X_i) = 0$
and so that each $X_i$ satisfies a Lip-lip condition of the form \eqref{eq_Liplip} with constant $M = \d_i$.
\end{enumerate}
\end{thm}

Inspired by Bate's work, we also discuss how the proof of Theorem \ref{thm_converse} extends to show the same result.  We also show a stronger characterisation of measurable differentiable structures in terms of derivations, given later as Proposition \ref{prop_MDSderivs}.

\subsection{Plan of the paper and acknowledgments}
Section \S\ref{sect_prelims} reviews standard facts about doubling measures, Lipschitz functions, and measurable differentiable structures; experts can skip this part.  We discuss derivations in Section \S\ref{sect_derivs} and give a quick proof of (\ref{thm_converse}.B) $\Rightarrow$ (\ref{thm_converse}.A).  Here we also discuss tangent measures from  geometric measure theory and present a new construction for derivations with respect to them.

Section \S\ref{sect_eucl} contains a version of Theorem \ref{thm_converse} for Euclidean spaces and its proof; the key step lies in reducing the class of admissible functions for the Lip-lip condition, thereby reducing it to a geometric problem.  The case of metric spaces with doubling measures is treated in Section \S\ref{sect_metric}, where we implement the idea of passing to local coordinates.  Section \S\ref{sect_bate} is a short appendix, where we discuss Bate's theorem and prove a new characteristion for measurable differentiable structures.

The author would like to thank David Bate, Estibalitz Durand Cartagena, Juha Kinnunen, and Marta Szuma\'nska for their helpful comments, as well as Lizaveta Ihnatsyeva, Jes\'us A.\ Jaramillo, and Pekka Pankka for discussions that led to improvements in this work.  He lastly acknowledges the organisers of the 6th ECM Satellite Conference on Fourier Analysis and Pseudo-Differential Operators, held at Espoo, Finland in June 2012, who provided a hospitable setting that facilitated some of these discussions.

\section{Preliminaries} \label{sect_prelims}

Here and everywhere, $(X,d)$ always denotes a metric space.  When the metric is understood, we write $X = (X,d)$.  A {\em metric measure space} $(X,d,\mu)$ simply refers a metric space $(X,d)$ with a fixed choice of a Radon measure $\mu$, that is:\ $\mu$ is Borel regular and all balls with positive radius have finite, positive $\mu$-measure.

\subsection{Lipschitz functions}

The Lipschitz constant of $f: X \to \R$ is denoted as
$$
L(f) \;:=\; \sup\left\{ \frac{|f(y) - f(x)|}{d(x,y)} \,:\, x \neq y \text{ in } X \right\}.
$$
and we will consider various classes of such functions:
\begin{eqnarray*}
\Lip(X) &:=& \{ f : X \to \R \,;\, L(f) < \infty \} \\
{\Lip}_b(X) &:=& \{ f \in \Lip(X) \,;\, \|f\|_\infty < \infty \}.
\end{eqnarray*}
We now recall that $\Lip_b(X)$ is not only a Banach space, but a dual Banach space \cite{Arens:Eells}; see also \cite{Weaver:alg}.

\begin{lemma}[Arens-Eells] \label{lemma_dualbanach}
If $X$ is a metric space, then 
$\Lip_b(X)$ is (isometrically isomorphic to) a dual Banach space with respect to the norm
$$
\|f\|_{\Lip} \;:=\; \max\{ L(f), \|f\|_\infty \}.
$$
Moreover, on bounded subsets of $\Lip_b(X)$, the topology of weak-$*$ convergence agrees with that of pointwise convergence.
\end{lemma}

In order to exploit properties of the weak-star topology later, we now study an explicit predual space.  The discussion below essentially follows \cite[Chap.\ 2]{Weaver:alg}.

\subsubsection{A predual of $\Lip_b(X)$} \label{subsubsect_predual}
Given a metric space $X = (X,d)$, the function
$$
d_2(x,y) \;:=\; \min\{d(x,y),2\}
$$
is also a metric on $X$, which we write as $X_2 := (X,\rho_2)$.  By abstractly extending the space by one more point, written $X_2^+ \;:=\; X_2 \cup \{ e\}$, the metric also extends:
$$
d_2^+(x,y) \;:=\;
\begin{cases}
d_2(x,y), & \text{for } x \neq e \text{ and } y \neq e \\
2, & \text{for } x \neq y = e.
\end{cases}
$$
By \cite[Prop 1.7.1 \& 1.7.2]{Weaver:alg}, the space $\Lip_b(X)$ is isometrically isomorphic to
$$
{\Lip}_0(X_2^+) \;=\;
\{ f \in \Lip(X_2^+) \,;\, f(e) = 0 \}
$$
which is equipped with the Lipschitz constant (functional) as a norm:
\begin{equation} \label{eq_weirdnorm}
f \;\mapsto\; L(f|_{X_2^+}) \;:=\; 
\sup\left\{ \frac{|f(y)-f(x)|}{d_2^+(x,y)} \,;\, x,y \in X_2^+, x \neq y \right\}.
\end{equation}
It is clear that the inclusion map $f \in \Lip_b(X) \mapsto f \in \Lip_0(X_2^+)$ is well-defined.

Now define $\widetilde{AE}(X_2^+)$ as the set of so-called ``molecules'' \cite[Defn 2.2.1]{Weaver:alg} on $X_2^+$, i.e.\ functions supported on finite sets and of the form
\begin{equation} \label{eq_molecule}
m \,:=\,
\sum_{i=1}^n a_i (\chi_{\{x_i\}} - \chi_{\{y_i\}}), 
\end{equation}
for $(a_i)_{i=1}^n \subset \R$ and $(x_i)_{i=1}^n, (y_i)_{i=1}^n \subset X_2^+$. It admits a norm \cite[Cor 2.2.3(b)]{Weaver:alg}
$$
\|m\|_{AE} \;:=\; \inf\left\{
\sum_{i=1}^n |a_i| d_2^+(x_i,y_i) \,;\,
m \,=\, \sum_{i=1}^n a_i (\chi_{\{x_i\}} - \chi_{\{y_i\}})
\right\}
$$
and the {\em Arens-Eells space} $AE(X_2^+)$ of $X_2^+$ is defined as the norm-completion of $\widetilde{AE}(X_2^+)$ with respect to $\|\cdot\|_{AE}$.  It is thus a fact \cite[Thm 2.2.2]{Weaver:alg} that
\begin{equation} \label{eq_predual}
\big[ AE(X_2^+) \big]^* \;\cong\; {\Lip}_0(X_2^+) \;\cong\; {\Lip}_b(X)
\end{equation}
where the duality for $m \in \tilde{AE}(X_2^+)$ and $f \in \Lip_b(X)$ is given by
$$
\langle f, m \rangle \;:=\; \sum_{x \in X_2^+} m(x) \, f(x).
$$

\subsection{Differentiability on metric spaces} 
\label{subsect_lipconst}

Motivated by Rademacher's theorem, we now give a precise notion of differentiability on metric spaces.

\begin{defn} \label{defn_MDS}
Let $(X,d)$ be a metric space with a Radon measure $\mu$.

\begin{enumerate}
\vspace{.05in}
\item[(\ref{defn_MDS}.A)]
Let $\xi : X \to \R^N$ be Lipschitz and let $x \in X$.  
A function $f \in \Lip(X)$
is {\em differentiable at $x$ with respect to $\xi$} if there exists ${\bf v} \in \R^N$ so that
$$
\limsup_{y \to x} \frac{|f(y) - f(x) - {\bf v} \cdot\{ \xi(y) - \xi(x)\} |}{d(x,y)} \;=\; 0.
$$
\item[(\ref{defn_MDS}.B)]
A measurable subset $Y \subseteq X$ is a {\em chart (of differentiability)} if $\mu(Y) > 0$ and if there is a Lipschitz map $\xi: Y \to \R^N$,  called {\em (a choice of) coordinates} on $Y$,
with the following property:\ for every 
$f \in \Lip(X)$,
there is a unique measurable vectorfield ${\bf D}f : Y \to \R^N$ so that $f$ is differentiable at $\mu$-a.e.\ $x \in Y$, with ${\bf v} = {\bf D}f(x)$.

\vspace{.05in}
\item[(\ref{defn_MDS}.C)]
A space $(X,d,\mu)$ has a {\em measurable differentiable structure} (or {\em MDS}) if there is a collection of charts $\{X_m\}_{m=1}^\infty$, with coordinates $\xi^m : X \to \R^{N_m}$, so that
$$
\mu\Big( X \setminus \bigcup_{i=1}^\infty X_m \Big) \;=\; 0 
$$
in which case $\{(X_m,\xi^m)\}_{m=1}^\infty$ is called an {\em atlas} of $X$ and the associated vectorfields,  denoted by ${\bf D}^mf$, are called {\em measurable differentials} of $f$.  

Moreover, such a structure is called {\em $N$-dimensional} if $N = \sup_m N_m$ and it is {\em nondegenerate} if $N_m \geq 1$ holds for all $m \in \N$.
\end{enumerate}
\end{defn}

\begin{rmk}
For $N \in \N$, condition (\ref{defn_MDS}.C) is also known as a {\em strong measurable differentiable structure} in \cite{Keith}.
In contrast to other discussions \cite{Kleiner:Mackay}, \cite{Schioppa}, Definition \ref{defn_MDS} allows for infinite-dimensional measurable differentiable structures, or $N = \infty$, though each chart $X_m$ must still have a fixed dimension $N_m \in \N \cup \{0\}$.  Such spaces are also called {\em (Lipschitz) differentiability spaces} in \cite{Bate:Speight}, \cite{Bate}.
\end{rmk}

Related to the notion of measurable differentiable structure, the {\em variation} of $f: X \to \R$ at $x \in X$ is defined as
$$
L(f;x,r) \;:=\; \sup\Big\{ \frac{|f(y) - f(x)|}{r} \,:\, y \in \bar{B}(x,r) \Big\}
$$
and the {\em lower} and {\em upper pointwise Lipschitz constants} of $f$ at $x$ are defined as
\begin{eqnarray*}
\lip[f](x) &:=& \liminf_{r \to 0} L(f;x,r) \\
\Lip[f](x) &:=& \limsup_{r \to 0} L(f;x,r) \;=\; \limsup_{y \to x} \frac{|f(y) - f(x)|}{d(x,y)}.
\end{eqnarray*}
For $X = (\R^n,|\cdot|,\H^n)$, we have $\Lip[f](x) = |\nabla f(x)|$ whenever $\nabla f(x)$ is well-defined. 

\begin{rmk} \label{rmk_lipderiv1side}
Regarding differentiability and pointwise Lipschitz constants, first fix a Lipschitz map $\xi: X \to \R^N$.
\begin{enumerate}
\item[(\ref{rmk_lipderiv1side}.A)]
The differentiability of $f \in \Lip(X)$ at a point $x \in X$ with respect to $\xi$, in the sense of (\ref{defn_MDS}.A), is equivalent to the condition that
$$
\Lip[f - {\bf D}f(x) \cdot \xi](x) \;=\; 0.
$$
\item[(\ref{rmk_lipderiv1side}.B)]
Note that $f \mapsto \Lip[f](y)$ is a semi-norm when $y \in X$ is fixed.  It follows that if $f \in \Lip(X)$ is differentiable at $x \in X$ with respect to $\xi$, again in the sense of (\ref{defn_MDS}.A), then the following inequality holds:
$$
\Lip[f](x) \;\leq\;
L(\xi) \, |{\bf D}f(x)|.
$$
\end{enumerate}
\end{rmk}

%

Recalling Keith's theorem, the Lip-lip condition \eqref{eq_Liplip} with respect to a doubling measure on a metric space guarantees the existence of an MDS on that space.  Towards Theorem \ref{thm_converse}, however, we begin with spaces supporting such structures and then give a simpler criterion for checking the Lip-lip condition on them. 

\begin{lemma} \label{lemma_liplinear}
Let $\mu$ be a Radon measure on a metric space $(X,d)$ that satisfies the Lebesgue differentiation theorem.  If $(Y,\xi)$ is a chart of differentiability for $(X,d,\mu)$, then the following conditions are equivalent:\
\begin{enumerate}
\item[(\ref{lemma_liplinear}.A)] the subspace $(Y,d,\mu)$ satisfies the Lip-lip condition;
\item[(\ref{lemma_liplinear}.B)] inequality \eqref{eq_Liplip} holds $\mu$-a.e.\ on $Y$ for the sub-class of Lipschitz functions 
$$
\{ \ell \circ \xi \,:\, \ell : \R^{N_m} \to \R \, \text{ is affine } \}.
$$
\end{enumerate}
Moreover, the constants for \eqref{eq_Liplip} depend only on the chart $Y$.
\end{lemma}

\begin{proof}
Clearly (\ref{lemma_liplinear}.A) implies (\ref{lemma_liplinear}.B), with the same constant $M$.  For the other direction, fix $\e > 0$ and choose a sequence $(r_j)_{j=1}^\infty \subset \R^+$ with $r_j \searrow 0$ and so that
$$
\lim_{j \to 0} L(f;y,r_j) \;\leq\; \lip[f](y) \;+\; \e
$$
holds.  Using Definition \ref{defn_MDS}, we then estimate
\begin{equation} \label{eq_lipsequence}
\left.
\hspace{.6in}
\begin{split}
\lip[f](y) + \e &\;\geq\;
\lim_{j \to \infty} L(f;y,r_j) \\ &\;=\;
\lim_{j \to \infty} \left\{
L(f;y,r_j) +
L\big( f - {\bf D}f(y) \cdot \xi; y, r_j \big)
\right\} \\ &\;\geq\;
\lim_{j \to \infty}
L\big( {\bf D}f(y) \cdot \xi; y, r_j \big) \\ &\;\geq\;
\lip[ {\bf D}f(y) \cdot \xi ](y).
\end{split}
\hspace{.1in}
\right\}
\end{equation}
So if (\ref{lemma_liplinear}.B) holds with constant $M$, then as $\e \to 0$, Condition (\ref{lemma_liplinear}.A) follows from (\ref{rmk_lipderiv1side}.B), with constant $M\sqrt{N}$.
\end{proof}

\begin{rmk} \label{rmk_liplinear}
In the proof above, note that the differentiability property (\ref{defn_MDS}.A) is used, but not the uniqueness of measurable differentials from (\ref{defn_MDS}.B).
\end{rmk}

\subsection{Measures of controlled growth}

Let $\mu$ be a {\em doubling} measure on $X$ -- that is, $\mu$ is Radon and 
there exists $\kappa \geq 1$ so that
\begin{equation} \label{eq_doubling}
0 \;<\; \mu(B(x,2r)) \;\leq\; \kappa \, \mu(B(x,r)) \;<\; \infty
\end{equation}
holds, for all $x \in X$ and $r \in (0,{\rm diam}(X))$.
Metric spaces with such measures are also known as {\em spaces of homogeneous type}, after Coifman and Weiss \cite{Coifman:Weiss}.

\begin{rmk} \label{rmk_doubling}
We briefly list several useful properties of such measures.
\begin{enumerate}
\item[(\ref{rmk_doubling}.A)]
If $\mu$ is doubling on $X$ with constant $\kappa$, then \eqref{eq_doubling} also holds for balls with any center in $B(x,2r)$.  Indeed, it is known that for each $R > 0$ we have
$$
\mu(B(x,R)) \;\leq\; \left( \frac{r}{2R} \right)^{\log_2(\kappa)} \mu(B(y,r))
$$
for all $y \in B(x,R)$ and all $0 < r < 2R$; see \cite[Eq.\ 4.16]{Heinonen:LLA}.

\vspace{.05in}
\item[(\ref{rmk_doubling}.B)]
If $\mu$ is doubling on $X$ with constant $\kappa$, then $(X,d)$ is also a {\em doubling space}; in other words, there exists $N = N(\kappa) \in \N$ so that every ball $B(x,r)$ in $X$ can be covered by $N$ balls with centers in $B(x,r)$ and with radius $\frac{r}{2}$.  In particular, every ball in $X$ is {\em totally bounded}, so if $X$ is complete, then closed balls in $X$ are compact.

Moreover, such measures $\mu$ have the Vitali covering property \cite{Coifman:Weiss} and therefore satisfy the {\em Lebesgue differentiation theorem}, that is:\
\begin{equation} \label{eq_lebdiff}
\frac{1}{\mu(B(x,r))} \int_{B(x,r)} h \,d\mu \;\to\; h(x)
\end{equation}
holds for all $h \in L^1(X,\mu)$, at $\mu$-a.e.\ $x \in X$.

\vspace{.05in}
\item[(\ref{rmk_doubling}.C)]
If $(X_m,\xi^m)$ is a chart on $(X,d,\mu)$ with $\mu$ doubling, then the components of $\xi^m$ can be chosen as distance functions \cite[Cor 6.30]{Schioppa}, i.e.\ there exist points $(z_i^m)_{i=1}^{N_m} \subset X$ so that
\begin{equation} \label{eq_distcoords}
\xi^m_i(x) \,:=\, d(x,z_i^m).
\end{equation}
and $\xi^m := (\xi^m_i)_{i=1}^{N_m}$ satisfies Definition \ref{defn_MDS} for all Lipschitz functions on $X$.
\end{enumerate}
\end{rmk}

More generally, Keith considers also chunky measures \cite[Defn 2.2.1]{Keith}.  On doubling metric spaces in the sense of (\ref{rmk_doubling}.B), the Lip-lip condition with respect to such measures is also sufficient for Rademacher-type theorems \cite[Thm 2.3.1]{Keith}.

\begin{defn}[Keith] \label{defn_chunky}
A Radon measure $\mu$ on $X$ is {\em chunky} if for $\mu$-almost every $x \in X$, there exist $(r_n)_{n=1}^\infty$ in $\R^+$ with $r_n \searrow 0$ and with the property that, for every $\e > 0$ there exists $N \in \N$ satisfying the inequality
$$
\mu( B(x,r_n) ) \;<\; N \, \mu( B(y, \e r_n) )
$$
for all $n \geq N$ and all $y \in B(x,r_n)$.
\end{defn}

It is clear from (\ref{rmk_doubling}.A) that every doubling measure is chunky.  The next lemma takes a similar direction, by combining some of the previous observations.

\begin{lemma} \label{lemma_relchunky}
Let $\mu$ be a doubling measure on $X$ and let $A \subseteq X$.  If $\mu(A) > 0$, then the restriction measure 
$\mu|_A(S) := \mu(A \cap S)$ is chunky.
\end{lemma}

\begin{proof}
Indeed, (\ref{rmk_doubling}.B) implies that for $\mu$-a.e.\ $x \in A$, there exists $\rho_x > 0$ so that
$$
\frac{\mu(A \cap B(x,r))}{\mu(B(x,r))} \;\geq\; \frac{1}{2}
$$
holds whenever $r \in (0,\rho_x)$, so $\mu|_A$ satisfies the doubling condition \eqref{eq_doubling} with constant $2\kappa$ in place of $\kappa$, for all balls with centers in $B(x,\frac{\rho_x}{2})$ and radii at most $\frac{\rho_x}{2}$.
In particular, $\mu|_A$ satisfies the property in (\ref{rmk_doubling}.A) and is therefore chunky.
\end{proof}

\section{Derivations, pushforwards, and Euclidean tangents} \label{sect_derivs}

We now consider generalised differential operators called {\em (metric) derivations}.  The following notion is due to Weaver \cite[Defn 21]{Weaver} and holds in the general setting of measure spaces that support {\em measurable metrics}.  For the specific setting of metric measure spaces, see \cite[\S13]{Heinonen}, \cite{Gong_rigidity}, \cite{Gong_diffstruct}, and \cite{Schioppa}.

\begin{defn}[Weaver] \label{defn_derivation}
Fix a Borel measure $\mu$ on a metric space $(X,d)$.  A {\em derivation}
$\d : {\Lip}_b(X) \to L^\infty(X,\mu)$
is a bounded linear operator that obeys
\begin{enumerate}
\item[(\ref{defn_derivation}.A)] the product rule:\ $\d (fg) \;=\; f \, \d g \,+\, g \,\d f$;
\item[(\ref{defn_derivation}.B)] weak continuity:\ if $(f_j)_{j=1}^\infty$ is bounded in $\Lip_b(X)$ and converges pointwise to $f$, then 
$(\d f_j)_{j=1}^\infty$ converges weak-star to $\d f$ 
in $L^\infty(X,\mu)$.
\end{enumerate}
The space of derivations on $(X,d,\mu)$ is denoted by $\U(X,\mu)$, and the operator norm of $\d \in \U(X,\mu)$ is denoted
$$
\|\d\|_{\rm op} \;:=\; \sup\left\{ \|\d f\|_{L^\infty(X,\mu)} \,:\, f \in {\Lip}_b(X),\, \|f\|_{\Lip} \leq 1 \right\}.
$$
\end{defn}

\subsection{Derivations} \label{subsect_derivsbasics}

Note that $\U(X,\mu)$ forms a module over $L^\infty(X,\mu)$ via the action
$$
(\lambda \, \d)f(x) \;:=\; \lambda(x) \, \d f(x),
$$
so notions of {\em linear independence}, {\em basis}, and {\em rank} are well-defined for derivations.  In particular, characteristic functions $\chi_A$ of positive $\mu$-measured subsets $A \subset X$ induce an action of {\em locality} \cite[Thm 29]{Weaver} on $\U(X,\mu)$.

\begin{lemma}[Weaver] \label{lemma_locality}
Let $(X,d,\mu)$ be a metric measure space with $A \subseteq X$.  Then
$$
\U(A,\mu) \;=\;
\left\{ \chi_A\d \,:\, \d \in \U(X,\mu) \right\}.
$$
\end{lemma}

As a result, for Radon measures $\mu$ on $X$, the action of $\d \in \U(X,\mu)$ on $f \in \Lip(X)$ is well-defined, in that on every ball $B \subset X$, we interpret $\d f$ as
$$
(\d f)|_B \;=\; \chi_B\d(f|B)
$$
This implies, moreover, that sharper estimates hold for $\d f(x)$.  Indeed, for every $f \in \Lip(X)$, $x \in X$, and $r > 0$, the auxiliary function
$$
f_r \;:=\; (f - f(x))\big|_{B(x,r)}
$$
satisfies $\|f_r\|_\infty \leq r$ and $L(f_r) \leq L(f)$ and $\d f_r = \d f$ on $B(x,r)$.  So for $\mu$-density points $x \in X$ and sufficiently small $r > 0$, we obtain
\begin{equation} \label{eq_derivop}
\left.
\hspace{.225in}
\begin{split}
|\d f(x)| &\;=\; 
|\d f_r(x)| \;\leq\; 
\|\d f_r\|_{L^\infty(X,\mu)} \\ &\;\leq\;
\|\d\|_{\rm op} \|f_r\|_{\Lip} \;=\;
\|\d\|_{\rm op} \max\{ \|f_r\|_\infty, L(f_r) \} \;\leq\;
\|\d\|_{\rm op} L(f).
\end{split}
\hspace{.025in}
\right\}
\end{equation}

What follows is a characterisation theorem for measurable differentiable structures from \cite[Thm 1.6]{Gong_diffstruct}; see also \cite[Thm 5.9]{Schioppa}.  The proof uses a rank bound for derivations with respect to doubling measures \cite[Lem 1.10]{Gong_diffstruct}, as stated below as a lemma.

\begin{lemma} \label{lemma_doublingrank}
Let $(X,d)$ be an $N$-doubling metric space for some $N \in \N$.  Then there exists $N = N(\kappa) \in \N$ so that $\U(X,\mu)$ has rank at most $N$, for every Radon measure $\mu$ on $X$.
\end{lemma}

\begin{thm} \label{thm_MDSchar}
Let $(X,d)$ be a metric space with a doubling measure $\mu$.  If $\{X_m\}_{m=1}^\infty$ are subsets of $X$ with $\mu(X \setminus \bigcup_{m=1}^\infty X_m) = 0$, then the following are equivalent:\
\begin{itemize}
\item[(\ref{thm_MDSchar}.A)]
$X$ supports an $N$-dimensional measurable differentiable structure for some $N \in \N$, with charts $\{(X_m,\xi^m)\}_{m=1}^\infty$;
\item[(\ref{thm_MDSchar}.B)]
for each $m \in \N$, there exist a constant $K_m \geq 1$ and a linearly independent set ${\bf d}^m = (\d_i^m)_{i=1}^{N_m}$ in $\U(X_m,\mu)$ with $N_m \in N$ and so that the inequality
\begin{equation} \label{eq_lipderiv}
K_m^{-1}\Lip[f](x) \;\leq\; |{\bf d}^mf(x)| \;\leq\; K_m \Lip[f](x)
\end{equation}
holds for all $f \in \Lip(X)$ at $\mu$-a.e.\ $x \in X_m$.
\end{itemize}
\end{thm}

\begin{rmk} \label{rmk_lipderiv}
In the above theorem, the tuple of derivations agrees with the measurable differential, i.e.\ ${\bf d}^m := {\bf D}^m $, and the doubling condition is used only to check that each component of ${\bf D}^m $ is weakly continuous, hence a derivation.

Inequality \eqref{eq_lipderiv} in fact holds for all metric spaces supporting an MDS, even when the Radon measure $\mu$ is not doubling; for details, see \cite[Lem 5.1]{Gong_diffstruct}.
\end{rmk}

For completeness, we now sketch one of the implications in Theorem \ref{thm_converse}, since Theorem \ref{thm_keith} does not automatically apply to it.

\begin{proof}[Proof of (\ref{thm_converse}.B) $\Rightarrow$ (\ref{thm_converse}.A)]
Up to a subset of $\mu$-measure zero, the union of the subsets $\{Y_m\}_{m=1}^\infty$ covers $X$; without loss, each $Y_m$ has positive $\mu$-measure.
Since $\mu$ is doubling for some $\kappa \geq 1$, it follows by Lemma \ref{lemma_relchunky} that 
$\mu_m := \mu|_{Y_m}$ is chunky; in fact, the proof of that lemma shows that $\mu_m$ is locally doubling with constant $2\kappa$.

By hypothesis, each $Y_m$ satisfies a Lip-lip condition.  As indicated before, Keith's theorem applies to this case, so each $Y_m$ has an MDS with atlas $\{X_{ml}\}_{l=1}^\infty$.  Further applying Theorem \ref{thm_MDSchar}, each chart $X_{ml}$ supports a basis in $\U(X,\mu_l)$.  Because $\mu_m$ is locally doubling with constant $2\kappa$, a standard Vitali covering argument and Lemmas \ref{lemma_locality} and \ref{lemma_doublingrank} imply that the MDS on $X_{ml}$ is at most $N(\kappa)$-dimensional.

Thus the full union $\{ X_{ml} \}_{m,l=1}^\infty$ forms an atlas for $X$.
\end{proof}

\subsection{Pushforwards} \label{subsect_pushfwd}

For a Borel map $T : X \to Y$ between metric spaces, every Radon measure $\mu$ on $X$ admits a {\em pushforward measure} $T_\#\mu$ on $Y$, 
\begin{equation} \label{eq_pushfwdmeas}
T_\#\mu(A) \;:=\; \mu(T^{-1}(A))
\end{equation}
which is Radon and obeys the following transformation formula \cite[Thm 1.18 \& 1.19]{Mattila}, for all Borel $\varphi : Y \to \R$:
\begin{equation} \label{eq_transform}
\int_Y \varphi \, d(T_\#\mu) \;=\; \int_X (\varphi \circ T) \, d\mu
\end{equation}
As shown in \cite[Lem 2.17]{Gong_rigidity}, for every $\d \in \U(X,\mu)$ there is a unique {\em pushforward derivation} $T_\#\d \in \U(Y,\zeta_\#\mu)$ that is completely determined by the formula
\begin{equation} \label{eq_derivpushfwd}
\int_Y \psi \, [T_\#\d]f \, d(T_\#\mu) \;=\; 
\int_X (\psi \circ T) \, \d(f \circ T) \,d\mu
\end{equation}
for all $\psi \in L^1(Y,T_\#\mu)$ and $f \in \Lip(X)$, and the linear operator 
$$
\d \;\mapsto\; T_\#\d
$$
preserves linear independence \cite[Lem 2.18]{Gong_rigidity}. 

Moreover, $(T_\#\d)f \circ T$ and $\d(f \circ T)$ agree as dual elements acting on the class of composite functions
$\{ \psi \circ T : \psi \in L^1(Y,T_\#\mu)\}$.
For spaces supporting MDS's with $T = \xi^m$, however, they are equal in the usual sense.

\begin{lemma} \label{lemma_pushfwd}
Let $(X,d)$ be a metric space with doubling measure $\mu$.  If $X$ supports a measurable differentiable structure with charts $\{ (X_m,\xi^m) \}_{m=1}^\infty$, then
$$
[\xi^m_\#{\bf D}^m]f \circ \xi^m \;=\; {\bf D}^m (f \circ \xi^m)
$$
holds $\mu$-a.e.\ on each $X_m$ for all $f \in \Lip(X)$.
\end{lemma}

\begin{proof}
Fix $\d := {\bf D}^m _i$ for some $i \in \{1, \cdots, N_m\}$, where ${\bf D}^m  = ({\bf D}^m _1, \cdots , {\bf D}^m _{N_m})$.  The previous transformation formulas \eqref{eq_transform} and \eqref{eq_derivpushfwd} imply, in particular, that
\begin{equation} \label{eq_transform2}
\left.
\hspace{.45in}
\begin{split}
\int_X (\psi \circ \xi^m) \big( [\xi^m_\#\d]f \circ \xi^m\big) \,d\mu &\;=\; 
\int_Y \psi \, [\xi^m_\#\d]f \,d(\xi^m_\#\mu) \\ &\;=\; 
\int_X (\psi \circ \xi^m) \d(f \circ \xi^m) \,d\mu
\end{split}
\hspace{.45in}
\right\}
\end{equation}
holds for all $f, \psi \in \Lip(\R^{N_m})$, with $\psi$ compactly supported.  As a shorthand, put
\begin{eqnarray*}
F_m &:=& \d(f \circ \xi^m) - [\xi^m_\#\d]f \circ \xi^m, \\
c_m &:=& \big\| |\d(f \circ \xi^m)| + |[\xi^m_\#\d]f \circ \xi^m| \big\|_{L^\infty(X,\mu)}.
\end{eqnarray*}
Given $h \in L^1(X,\mu)$ and $\e > 0$, 
since $\mu$ is doubling,
there exists $h' \in \Lip_b(X)$, constructed via Lipschitz partitions of unity \cite[p.\ 1908]{Franchi:Hajlasz:Koskela}, so that
\begin{equation} \label{eq_approx2}
\|h-h'\|_{L^1(X,\mu)} \;<\; \frac{\e}{2c_m}.
\end{equation}
So for $\mu$-a.e.\ $x \in {\rm spt}(h')$ and for the affine function $l_m^x : \R^{N_m} \to \R$, given by
$$
l_m^x(z) \;:=\; 
h'(x) \,-\, {\bf D}^m h'(x) \cdot \big( z - \xi^m(x) \big),
$$
Equation \eqref{eq_transform2} and condition (\ref{rmk_lipderiv1side}.A) imply that, for sufficiently small $r = r(x) > 0$ and for the $L^1$-test function $\psi := \chi_{B(x,r)} (\ell^x_m \circ \xi^m)$, we have
\begin{eqnarray*}
\left| \int_{B(x,r)} h' F_m \,d\mu \right| &\leq&
\int_{B(x,r)} \big| (h' - l_m^x \circ \xi^m) \, F_m \big| \,d\mu \,+\,
\left| \int_{B(x,r)} (l_m^x \circ \xi^m) \, F_m \,d\mu \right| \\ &=&
\int_{B(x,r)} \big| (h' - l_m^x \circ \xi^m) \, F_m \big| \,d\mu \,+\, 0 \\ &\leq&
c_m\sup_{B(x,r)}|h' - l_m^x \circ \xi^m| \, \mu(B(x,r)) \;\leq\;
\frac{\e}{2} \, \frac{\mu(B(x,r))}{\mu\big({\rm spt}(h')\big)}.
\end{eqnarray*}
Lastly, by Vitali's Covering Theorem the collection of balls 
$$
\left\{ B(x,\rho) \,;\, x \in {\rm spt}(h'),\, 0 < \rho < r(x) \right\}
$$
contains a pairwise-disjoint sub-collection, denoted by $\{ B(x_i,r_i) \}_{i=1}^\infty$, so that
$$
\mu\Big( {\rm spt}(h') \setminus \bigcup_{i=1}^\infty B(x_i,r_i) \Big) \;=\; 0
$$
and hence the mean-value estimate becomes
\begin{equation} \label{eq_approx3}
\left| \int_X h' F_m \,d\mu \right| \;\leq\;
\sum_{i=1}^\infty \left| \int_{B(x_i,r_i)} h' F_m \,d\mu \right| \;\leq\;
\sum_{i=1}^\infty \frac{\e}{2} \frac{\mu(B(x_i,r_i))}{\mu\big({\rm spt}(h')\big)} \;\leq\;
\frac{\e}{2}.
\end{equation}
Since $\e > 0$ was arbitrary, the lemma follows from combining \eqref{eq_approx2} and \eqref{eq_approx3}.
\end{proof}

\subsection{Tangent measures and derivations}

Before moving to proofs of the main result and auxiliary lemmas,  we introduce a new construction for derivations in $\R^n$, as inspired by the work of Marstrand \cite{Marstrand} and Preiss \cite{Preiss}.
To begin, recall that for bounded domains $\Omega \subset \R^n$, the Riesz representation theorem states that the Banach dual of $C_b(\Omega)$, the class of bounded continuous functions on $\Omega$, consists of signed measures on $\Omega$ under the total variation norm:
$$
\|\mu\|_{\rm op} \;:=\;
\sup\left\{ 
\int_\Omega \varphi \,d\mu \;;\; \varphi \in C_b(\Omega), \, \|\varphi\|_\infty \,\leq\, 1
\right\}
$$
As a result, the class of Radon measures on $\Omega$ has a natural weak-star topology.

\begin{defn} \label{defn_tanmeas}
Let $\mu$ be a Radon measure on a bounded domain $\Omega \subset \R^n$ and let $a \in \Omega$.  A measure $\nu$ on $\R^n$ is called a {\em tangent measure} of $\mu$ at $a$, denoted $\nu \in \Tan(\mu,a)$, if there exist $(c_j)_{j=1}^\infty, (r_j)_{j=1}^\infty \subset \R^+$ with $r_j \searrow 0$ and so that
$$
\nu \;=\;
\wslim_{j \to \infty} c_j\big(T_{a,r_j}\big)_\#\mu, \, \text{ where } \,
T_{a,r}(x) \;:=\; \frac{x-a}{r}
$$
and where the limit is taken in the weak-star topology of signed measures.

A {\em tangent derivation} of $\mu$ at $a$ is a derivation in $\U(\R^n,\nu)$, for some $\nu \in \Tan(\mu,a)$.
\end{defn}

It is known \cite[Chap.\ 14]{Mattila} that if $\mu$ is Radon, then so is any $\nu \in \Tan(\mu,a)$.

Just as tangent measures arise from ``zooming in'' a measure at a fixed point, tangent derivations arise from the same zooming process at the same point.

\begin{thm} \label{thm_tanderiv}
Let $\Omega$ be a bounded domain in $\R^n$ and let $\mu$ be a Radon measure supported in $\Omega$.  If $a \in \Omega$ is a $\mu$-density point and if $\nu \in \Tan(\mu,a)$, 
then there exists a linear operator
$
T_a: \U(\Omega,\mu) \to \U(\R^n,\nu)
$
so that $\d \neq 0$ implies $T_a\d \neq 0$.
\end{thm}

To prove the theorem, we will require an auxiliary result, called a ``Chain Rule'' for derivations \cite[Lem 2.19]{Gong_rigidity}.

\begin{lemma} \label{lemma_chainrule}
Let $\nu$ be a Radon measure on $\R^n$.  For every $f \in \Lip(\R^n)$, there is a $\nu$-measurable ${\bf v}_f = (v_f^i)_{i=1}^n : \R^n \to \R^n$ with each $v_f^i \in L^\infty(\R^n,\mu)$ and so that
$$
\d f \;=\;
{\bf v}_f \cdot \d\id_{\R^n} \;=\;
\sum_{i=1}^n v_f^i \, \d x_i.
$$
holds, for all $\d \in \U(\R^n,\nu)$.  If $f \in C^1(\R^n)$, then ${\bf v}_f = \nabla f$.
\end{lemma}

As a warning, the proof of Theorem \ref{thm_tanderiv} is long and involved, so it is split into four steps for the convenience of the reader.  Step 3 is the most technical part, but the idea is simple:\ the ``zooming in'' process for tangent measures can be unraveled into a ``zooming out'' process for Lipschitz functions, which in turn is compatible with the weak-star topology of $\Lip_b(\Omega)$.  A careful argument using the predual $AE(X_2)$ explicitly ensures uniformity of the zooming process, so $[T_a\d]$ will be well-defined whenever $\d \in \U(\Omega,\mu)$.

\begin{proof}[Proof of Theorem \ref{thm_tanderiv}]
For $\nu \in \Tan(\mu,a)$ and $j \in \N$, let $(c_j)_{j=1}^\infty, (r_j)_{j=1}^\infty$ be its associated parameters as in Definition \ref{defn_tanmeas}, and put 
$$
\nu_j \;:=\; c_j(T_{a,r_j})_\#\mu.
$$
Since $C_b(\Omega)$ is separable, the weak-star topology of Radon measures is metrizable, so the sequence $(\nu_j)_{j=1}^\infty$ must be bounded in the total variation norm.

\vspace{.05in}
\noindent
{\em Step 1:\ Defining $T_a\d$}.\
For each $\d \in \U(\R^n,\mu)$, Lemma \ref{lemma_pushfwd} implies that
$$
\d_j \;:=\; r_j(T_{a,r_j})_\#\d
$$
is well-defined in $\U(\R^n,\nu_j)$, and for each $f \in \Lip(\R^n)$, we obtain a signed measure
$$
d\nu_{\d,j}(x) \;:=\; \d_jf(x) \,d\nu_j(x)
$$
with uniform bounds for the total variation norm.  To see this, letting $\varphi \in C^0_c(\R^n)$ with $\|\varphi\|_\infty \leq 1$ and applying \eqref{eq_derivop}, we estimate
\begin{eqnarray}
\left| \int_{T_{a,r_j}^{-1}(\Omega)} \varphi\, d\nu_{\d,j} \right| &=&
\left| \int_{T_{a,r_j}^{-1}(\Omega)} \varphi\, \d_jf \,d\nu_j \right| \;=\;
\notag
c_jr_j \left| \int_\Omega (\varphi \circ T_{a,r_j}) \d(f \circ T_{a,r_j}) \,d\mu \right| \\ &\leq&
\notag
c_jr_j \|\d(f \circ T_{a,r_j})\|_{L^\infty(X,\mu)} \int_\Omega |\varphi \circ T_{a,r_j}| \,d\mu \\ &\leq&
\label{eq_tanderiv1}
r_j \|\d\|_{\rm op} \, L(f \circ T_{a,r_j}) \int_{T_{a,r_j}^{-1}(\Omega)} |\varphi|\, c_j d(T_{a,r_j})_\#\mu \\ &\leq&
\notag
\|\d\|_{\rm op} \, L(f) \, r_j \, L(T_{a,r_j}) \int_{T_{a,r_j}^{-1}(\Omega)} |\varphi| \,d\nu_j \\ &\leq&
\label{eq_tanderiv2}
\|\d\|_{\rm op} \, L(f) \, \sup_j \|\nu_j\|_{\rm op} \;<\; \infty
\end{eqnarray}
where \eqref{eq_tanderiv1} follows from boundedness of $\d$ and \eqref{eq_tanderiv2} follows from $L(T_{a,r_j}) = r_j^{-1}$; taking suprema over $\|\varphi\|_\infty \leq 1$, we obtain the desired norm bound.

By weak-star compactness of signed Radon measures, there is a convergent subsequence $(\nu_{\d,j_k})_{k=1}^\infty$ with a weak-$*$ limit $\nu_\d$.  By similar estimates as above, 
$$
h \;\mapsto\; \int_{\R^n} h \,d\nu_\d
$$
is a well-defined element of $[L^1(\R^n,\nu)]^*$; since $\nu$ is Radon, we have 
$$
L^\infty(\R^n,\nu) \;\cong\; [L^1(\R^n,\nu)]^*
$$
and thus there is a unique $\lambda_{\d,f} \in L^\infty(\R^n,\nu)$ that satisfies
$d\nu_\d = \lambda_{\d,f} d\nu$.
The operator $[T_a\d] : \Lip(\R^n) \to L^\infty(\R^n,\nu)$ is thereby defined as
\begin{equation} \label{eq_tanderiv}
[T_a\d]f \;:=\; \lambda_{\d,f}.
\end{equation}
\noindent
{\em Step 2:\ For smooth $f$, sublimits are limits}.\
By iterating the argument in Step 1 with $h = x_i$ for $i = 1, 2, \ldots n$ and taking nested subsequences of $(j_k)_{k=1}^\infty$, we obtain well-defined functions $\{ ([T_a\d])x_i \}_{i=1}^n$ in $L^\infty(\R^n,\nu)$ via a {\bf fixed} subsequence of $(r_j)_{j=1}^\infty$.  With abuse of notation, the same symbols $(r_j)$ will denote this subsequence.  We also write ${\bf x} = \id_{\R^n}$ for short.

For $g \in C^1(\R^n)$, the Chain Rule (Lemma \ref{lemma_chainrule}) implies that
$\d_jg = \nabla g \cdot \d_j\x$
and hence, by approximation of $L^1(\R^n,\nu)$ with continuous functions, we have
\begin{equation} \label{eq_smoothchainrule}
[T_a\d] g \;=\; \nabla g \cdot [T_a\d]\x
\end{equation}
As a result, the RHS is independent of the choice of subsequence $(\nu_{\d,j_k})_{k=1}^\infty$ taken in the construction of $[T_a\d] g$.  It is not only a weak-star sublimit, but a full limit:
$$
[T_a\d] g \, d\nu \;=\; \wslim_{j \to \infty} \, \d_jg \, d\nu_j.
$$
As a consequence, $T_a\d$ is linear on $C^1(\R^n) \cap \Lip(\R^n)$, since
\begin{eqnarray*}
[T_a\d](g_1 + g_2) \, d\nu &=& 
\wslim_{j \to \infty} \, \d_j(g_1 + g_2) \, d\nu_j \\ &=&
\Big( \wslim_{j \to \infty} \, \d_jg_1 \, d\nu_j \Big) \,+\,
\Big( \wslim_{j \to \infty} \, [T_a\d] g_2 \, d\nu_j \Big) \\ &=&
([T_a\d] g_1 \,+\, [T_a\d] g_2) \, d\nu
\end{eqnarray*}
holds, under the topology of signed measures,
and it similarly satisfies the Leibniz rule for the same subclass of functions.

\vspace{.05in}
\noindent
{\em Step 3:\ Sublimits are always limits}.\
For nonsmooth $f \in \Lip(\R^n)$, let $t > 0$ and consider smooth, symmetric mollifiers $\eta_t : \R^N \to [0,\infty)$, supported on $\bar{B}(0,t)$, and put $f_t := f * \eta_t$.  Clearly $(f_t)_{t > 0}$  converges uniformly to $f_0 := f$, as
\begin{equation} \label{eq_unifconv}
|f(x) - f_t(x)| \;\leq\;
\int_{\R^n} |f(x) - f(y)| \eta_t(y) \,dy \;\leq\;
\sup_{B(x,t)} |f - f(x)| \;\leq\; L(f)\,t.
\end{equation}
Moreover, the sequence is uniformly $L(f)$-Lipschitz, with norm bounds
$$
\|\nabla f_t\|_{L^\infty(\R^n,\nu)} \;\leq\; L(f_t) \;\leq\; L(f) \,<\; \infty
$$
for all $t > 0$, so by weak-$*$ compactness in $L^\infty(\R^n,\nu)$, there exist a subsequence $(t_i)_{i=1}^\infty$ and a vectorfield ${\bf v}_f : \R^n \to \R^n$ so that $\nabla f_{t_i} \wsto {\bf v}_f$ in $L^\infty(\R^n,\nu)$.

\begin{claim} \label{claim_sublimit}
A Chain Rule holds for $T_a\d$:\ i.e.\ for all $f \in \Lip(\R^n)$, we have
$$
[T_a\d] f \;=\; {\bf v}_f \cdot [T_a\d]\x \; \; \nu\text{-a.e.\ on } \R^n.
$$
\end{claim}
Equivalently by \eqref{eq_smoothchainrule}, it suffices to show that in $L^\infty(\R^N,\nu)$,
\begin{equation} \label{eq_tanchainrule}
[T_a\d] f \;=\;
\wslim_{i \to \infty} \, [T_a\d] f_{t_i}.
\end{equation}
To this end, for $t \geq 0$ and $j \in \N$, we estimate
$$
r_j\left|f_t(T_{a,r_j}(x)) \,-\, f_t(T_{a,r_j}(y))\right| \;=\;
r_j\left|f_t\Big(\frac{x-a}{r_j}\Big) \,-\, f_t\Big(\frac{y-a}{r_j}\Big)\right| \;\leq\;
L(f) \, |x-y|
$$
so the sequence $\{ r_j (f_t \circ T_{a,r_j}) \}_{j=1}^\infty$ is $L(f)$-Lipschitz for every $t$, and hence bounded in $\Lip_b(\Omega)$.  Moreover, since $T_{a,r_j}: \R^n \to \R^n$ is bi-Lipschitz, it is clear that
\begin{equation} \label{eq_zoomoutfn}
\d\left[ r_j (f_t \circ T_{a,r_j}) \right] \;=\;
r_j \, [(T_{a,r_j})_\#\d]f_t \circ T_{a,r_j} \;=\;
(\d_jf_t) \circ T_{a,r_j}.
\end{equation}
Fixing $f_{t_0} := f$ for now, by Lemma \ref{lemma_dualbanach} and weak-star compactness of $\Lip_b(\R^n)$
there exists a subsequence of functions
$$
F_{0,k} \;:=\; r_{j_k} (f_{t_0} \circ T_{a,r_{j_k}})
$$
that converges to some $F_0$ in $\Lip_b(\Omega)$.  Similarly, from $\{r_{j_k} (f_{t_1} \circ T_{a,r_{j_k}})\}_{k=1}^\infty$ there is a weak-star convergent subsequence $\{F_{1,m}\}_{m=1}^\infty$ with limit $F_1$ in $\Lip_b(\Omega)$.

Proceeding by induction, there is a countable collection of weak-star convergent sequences $\{F_{i,m}\}_{m=1}^\infty$ with limit functions $F_i \in \Lip_b(\Omega)$, where $i = 0, 1, 2, \ldots$ and where the indices $m$ of the sequence $\{F_{i,m}\}_{m=1}^\infty$ arise from the indices $m'$ of the previous sequence $\{F_{i-1,m'}\}_{m'=1}^\infty$.  
\begin{claim} \label{claim_limitsoflimits}
$F_i \wsto F_0$ holds in $\Lip(\Omega)$.
\end{claim}

Indeed, for any $m \in \N$ with corresponding radii $r_m > 0$, inequality \eqref{eq_unifconv} gives
\begin{equation} \label{eq_unifconv2}
|F_{i,m} \,-\, F_{0,m}| \;\leq\;
r_m\,|(f_{t_i} \,-\, f) \circ T_{a,r_m}| \;\leq\;
\|f_{t_i} - f\|_\infty \;\leq\; L(f) \, t_i,
\end{equation}
so $F_{i,m} \to F_{0,m}$ is {\em uniformly} convergent in $\Omega$.  With the predual $AE(X_2^+)$ defined as in \S\ref{subsubsect_predual} and given $v \in AE(X_2^+)$ and $\e > 0$, choose $\tilde{v} \in \widetilde{AE}(X_2^+)$ of the form
$$
\tilde{v} \;=\; \sum_{i=1}^n a_i \, (\chi{\{x_i\}} - \chi{\{y_i\}})
$$
and which satisfies the norm bound
$\|v - \tilde{v}\|_{AE} < \frac{\e}{8L(f)}$.
Choose $i \in \N$ so that
$$
t_i \;<\; \Big(4L(f)\sum_{x \in X}|\tilde{v}(x)|\Big)^{-1}\e
$$
from which it follows from \eqref{eq_unifconv2} and the duality $\Lip_b(X) \cong [AE(X_2^+)]^*$ that
\begin{eqnarray*}
\sup_m |\langle \tilde{v}, F_{i,m} \,-\, F_{0,m} \rangle| &\leq&
\sum_{x \in X} |\tilde{v}(x)| \left\{ \sup_m \big| F_{i,m}(x) \,-\, F_{0,m}(x) \big| \right\} \\ &\leq&
L(f) t_i \sum_{x \in X} |\tilde{v}(x)| \;<\; \frac{\e}{4}
\end{eqnarray*}
With $i$ now fixed, now choose $m \in \N$ sufficiently large so that
$$
|\langle \tilde{v},\, F_i - F_{i,m} \rangle| \;\leq\; \frac{\e}{4} \text{ and }
|\langle \tilde{v},\, F_{0,m} - F_0 \rangle| \;\leq\; \frac{\e}{4}
$$
and hence Claim \ref{claim_limitsoflimits} follows from the above estimates and the Triangle inequality:
\begin{eqnarray*}
|\langle v,\, F_i - F_0 \rangle| &\leq&
|\langle v- \tilde{v},\, F_i - F_{i,m} \rangle| \,+\,
|\langle \tilde{v},\, F_i - F_{i,m} \rangle| \\ && \,+\,
|\langle \tilde{v},\, F_{i,m} - F_{0,m} \rangle| \,+\,
|\langle \tilde{v},\, F_{0,m} - F_0 \rangle| \\ &\leq&
\|v- \tilde{v}\|_{AE} \|F_i - F_{i,m}\|_{\Lip} \,+\, \frac{3\e}{4} \;\leq\; \e.
\end{eqnarray*}
Invoking weak continuity, each sequence $\{\d F_{i,k}\}_{k=1}^\infty$ converges weak-star to $\d F_i$ in $L^\infty(\R^n,\mu)$ and in turn, $\{\d F_i\}_{i=1}^\infty$ converges weak-star to $\d F_0$.  

Since $\mu$ is Radon and $\Omega$ is bounded, we have that for each $p \in (1,\infty)$,
$$
L^\infty(\Omega,\mu) \;\subset\; L^p(\Omega,\mu)
$$
and that $L^{p'}(\Omega,\mu)$ is dense in $L^1(\Omega,\mu)$, for $p' := \frac{p}{p-1}$.  It follows that the above sequences also converge weakly in $L^p(\Omega,\mu)$; by reflexivity for $1 < p < \infty$ and Mazur's lemma, there exist convex combinations $\{\d\tilde{F}_i\}_{i=1}^\infty$ of $\{\d F_i\}_{i=1}^\infty$ that converge in $L^p$-norm to $\d F_0$, so a subsequence (denoted with the same symbols) converges pointwise $\mu$-a.e.\ on $\Omega$.  The same functional analysis argument applies to each $i \in \N$, so there exist convex combinations $\{\d\tilde{F}_{i,m}\}_{m=1}^\infty$ which contain subsequences that converge $\mu$-a.e.\ on $\Omega$ to $\d\tilde{F}_i$.  

Let $\psi \in C_c(\Omega)$ and $\e > 0 $ be given and put $C_\psi := \|\psi\|_{L^1(\Omega,\mu)}$ for short.
By Egorov's theorem, apart from a subset $E \subset \Omega$ of $\mu$-measure at most 
$$
\mu(E) \;\leq\; \frac{\e}{16\|\psi\|_\infty\|\d\|_{\rm op}\, L(f)}
$$
the convergence $\d\tilde{F}_i \to \d F_0$ is uniform on $\Omega \setminus E$.  Choosing $i \in \N$ so that
$$
\| \d(F_0 - \tilde{F}_i) \|_{L^\infty(\Omega \setminus E,\mu)} \;\leq\; 
\frac{\e}{8C_\psi}
$$
with \eqref{eq_derivop} we may now estimate as follows:
\begin{eqnarray*}
\left| \int_{\R^n} \psi \, \d(F_0 - \tilde{F}_i) \,d\mu \right| &\leq&
\left| \int_{E} |\psi| \, \d(F_0 - \tilde{F}_i) \,d\mu \right| \,+\, 
\int_{\R^n \setminus E} |\psi| \, |\d(F_0 - \tilde{F}_i)| \,d\mu \\ &\leq&
2\|\psi\|_\infty\|\d\|_{\rm op}\, L(f) \, \mu(E) \,+\, 
C_\psi \| \d(F_0 - \tilde{F}_i) \|_{L^\infty(\Omega \setminus E,\mu)} \\ &\leq&
\frac{\e}{4}.
\end{eqnarray*}
The Egorov argument also applies to $\d\tilde{F}_{0,m} \to \d\tilde{F}_0$ and to $\d\tilde{F}_{i,m} \to \d\tilde{F}_i$, so with appropriate subsets $E_0, E_i \subset \Omega$ of small $\mu$-measure, we analogously obtain
$$
\left| \int_{\R^n} \psi \, \d(\tilde{F}_0 - \tilde{F}_{0,m}) \,d\mu \right| \;\leq\; \frac{\e}{8} 
\, \text{ and } \,
\left| \int_{\R^n} \psi \, \d(\tilde{F}_i - \tilde{F}_{i,m}) \,d\mu \right| \;\leq\; \frac{\e}{8}.
$$
So to prove Claim \ref{claim_sublimit}, let $\varphi \in C_c(\R^n)$ be arbitrary and choose $m \in \N$ so that, with the identity \eqref{eq_zoomoutfn}, we have
\begin{eqnarray*}
\left| \int_{\R^n} \varphi ([T_a\d] f - [T_a\d]\tilde{f}_{t_i}) d\nu \right| &\leq&
\left| \int_{T_{a,r_j}^{-1}(\Omega)} \varphi \, \d_m(f - \tilde{f}_{t_i}) d\nu_m \right| \,+\, \frac{\e}{4} \\ &=&
\left| c_m \int_\Omega (\varphi \circ T_{a,r_m}) \, \d(\tilde{F}_{0,m} - \tilde{F}_{i,m}) d\mu \right| \,+\, \frac{\e}{4}.
\end{eqnarray*}
where $(r_m)_{m=1}^\infty$ is the iterated subsequence of radii, associated to the construction of the $\{F_{i,m}\}$.  Putting $\psi_m := c_m (\varphi \circ T_{a,r_m})$ and recalling Definition \ref{defn_tanmeas}, by choosing $m$ larger as necessary, we have
$$
\frac{1}{2}\|\varphi\|_{L^1(\R^n,\nu)} \;\leq\;
\|\psi_m\|_{L^1(\Omega,\mu)} \;\leq\; 2\|\varphi\|_{L^1(\R^n,\nu)}
$$
Thus the previous estimates, with $\psi_m$ in place of $\psi$, come together as
\begin{equation*}
\begin{split}
&
\left| \int_{\R^n} \varphi ([T_a\d] f - [T_a\d]\tilde{f}_{t_i}) d\nu \right| \;\leq\;
\left| c_m \int_{\R^n} (\varphi \circ T_{a,r_m}) \, \d(\tilde{F}_{0,m} - \tilde{F}_{i,m}) d\mu \right| \,+\, \frac{\e}{4} \\
& \hspace{.4in} \;=\;
\left| \int_{\R^n} \psi_m \Big\{
\d(\tilde{F}_{0,m} - \tilde{F}_{i,m}) \,+\,
\d(\tilde{F}_{0,m} - \tilde{F}_{i,m}) \,+\,
\d(\tilde{F}_{0,m} - \tilde{F}_{i,m}) \Big\} d\mu \right| \,+\, \frac{\e}{4} \\
& \hspace{.4in} \;\leq\;
\frac{\e}{4} \,+\, \frac{\e}{4} \,+\, \frac{\e}{4} \,+\, \frac{\e}{4} \;=\; \e.
\end{split}
\end{equation*}
Since $\e > 0$ was arbitrary, we conclude that
$$
\lim_{i \to \infty} \int_{\R^n} \varphi ([T_a\d] f - [T_a\d]\tilde{f}_{t_i}) d\nu \;=\; 0
$$
and since $C_c(\R^n)$ is dense in $L^1(\R^n,\nu)$, Claim \ref{claim_sublimit} follows, with a modified subsequence $\{\tilde{f}_{t_i}\}_{i=1}^\infty$ in place of the original $\{f_{t_i}\}_{i=1}^\infty$ for \eqref{eq_unifconv2}.

\vspace{.05in}
\noindent
{\em Step 4:\ Each $T_a\d$ is a derivation}.\
By similar arguments as in Step 2, each $T_a\d$ is linear and satisfies the Leibniz rule.

As for weak continuity, let $(f_n)_{n=1}^\infty$ be a bounded sequence in $\Lip_b(\R^n)$ that converges pointwise to $f$, and let $\psi \in L^1(\R^n,\nu)$ and $\e > 0$ be arbitrary.  Since continuous functions are dense in $L^1(\R^n,\nu)$, there exists $\varphi \in C_c(\R^n)$ so that
$$
\|\psi - \varphi\|_{L^1(\R^n,\nu)} \;\leq\; \frac{\e}{3 \, \sup_{j \in \N} \|\d(f_j - f)\|_{L^\infty(\R^n,\nu)}}
$$
and for sufficiently large $j \in \N$, we have
$$
\left| \int [T_a\d](f_n - f) \, \varphi \,d\nu \right| \;\leq\; 
\left| \int \d_j(f_n - f) \, \varphi \,d\nu_j \right| \,+\, \frac{\e}{3}.
$$
Since $\d_j$ is a derivation, we already have $\d_jf_n \wsto \d_jf$ in $L^\infty(\R^n,\nu_j)$, so choose $n \in \N$ sufficiently large as to guarantee
$$
\left| \int \d_j(f_n - f) \, \varphi \,d\nu_j \right| \;\leq\; \frac{\e}{3}.
$$
Combining the last three estimates, the Triangle inequality implies that
$$
\left| \int [T_a\d](f_n - f) \, \psi \,d\nu \right| \;\leq\; \e,
$$
so $T_a\d$ is weakly continuous.  The theorem follows.
\end{proof}

Lastly, we note that the rank of derivation modules does not decrease under the process of taking tangent measures.  This relies on a criterion for linear independence of derivations \cite[Lem 2.12]{Gong_diffstruct}, of which one version is stated below.

\begin{lemma} \label{lemma_jacobian}
Let $\mu$ be a Radon measure on $\R^n$ and fix $(\d_i)_{i=1}^n \subset \U(\R^n,\mu)$.  If ${\bf d} := (\d_i)_{i=1}^n$ is linearly independent then the matrix-valued function 
$$
{\bf d}{\bf x} \;:=\; [\d_ix_k]_{i,k=1}^n
$$
is $\mu$-a.e.\ non-singular on $\R^n$. Moreover,
there exists a linearly independent set $\hat{\bf d} := (\hat\d_i)_{i=1}^n$ in $\U(\R^n,\mu)$ with the same span as ${\bf d}$ and is orthogonal in that 
$$
\text{if } i \,\neq\, k, \text{ then } \hat\d_ix_k \,=\, 0.
$$
\end{lemma}

The next result follows purely from the locality property (Lemma \ref{lemma_locality}) and unraveling previous definitions. Since the discussion has been technical so far, the argument has been added here for clarity.

\begin{cor} \label{cor_tanderivrank}
Let $\mu$ be Radon on $\R^n$ and fix a $\mu$-density point $a \in \R^n$.  If $\nu \in \Tan(\mu,a)$ and if $\U(\R^n,\mu)$ has rank $n$, then $\U(\R^n,\nu)$ also has rank $n$.
\end{cor}

\begin{proof}
Let $(\hat\d_i)_{i=1}^n$ be a linearly independent set in $\U(\R^n,\mu)$ as in Lemma \ref{lemma_jacobian}.  We may assume that $\hat\d_ix_i > 0$ holds $\mu$-a.e.\ on $X$, by replacing each $\hat\d_i$ with
$$
\big(\chi_{ \{ \hat\d_ix_i > 0\} } - \chi_{ \{ \hat\d_ix_i < 0\} }\big)\hat\d_i
$$
as necessary.  Now let $\varphi \in C_c(\R^n)$ be non-negative; if $i \neq k$, then 
\begin{eqnarray*}
\int_{\R^n} \varphi \, (T_a\hat\d_i)x_k \, d\nu &=&
\lim_{j \to \infty} \int_{\R^n} \varphi \, r_j\big((T_{a,r_j})_\#\hat\d_i\big)x_k \, d(T_{a,r_j})_\#\mu \\ &=&
\lim_{j \to \infty} \int_{\R^n} (\varphi \circ T_{a,r_j}) \, r_j\hat\d_i\Big(\frac{x_k - a}{r_j}\Big) d\mu \\ &=&
\lim_{j \to \infty} \int_{\R^n} (\varphi \circ T_{a,r_j}) \, \hat\d_ix_k \, d\mu \;=\; 0,
\end{eqnarray*}
so by density of continuous functions in $L^1(\R^n,\nu)$, it follows that
$$
(T_a\hat\d_i)x_k \;=\; 0 \, \text{ $\mu$-a.e.\ on $\R^n$, whenever } \, i \neq k.
$$
The $\nu$-a.e.\ inequality $(T_a\hat\d_i)x_i > 0$ follows from a similar computation as above.
Note that if $\{\lambda_i\}_{i=1}^n \subset L^\infty(\R^n,\nu)$ satisfies $\sum_i \lambda_i (T_a\hat\d_i) = 0$, then for each $k$,
$$
0\;=\; \sum_i \lambda_i (T_a\hat\d_i)x_k \;=\;
\sum_i \lambda_i (T_a\hat\d_i)x_k \;=\;
\lambda_k \, (T_a\hat\d_k)x_k
$$
holds $\nu$-a.e.\, so $\lambda_k = 0$; as a result, $(T_a\hat\d_i)_{i=1}^n$ must be linearly independent.

It is already known that every set of $n+1$ derivations on $\R^n$ is linearly dependent for any Borel measure \cite[Lem 2.13]{Gong_diffstruct}, so the lemma follows.
\end{proof}

It would be interesting to study analogues of tangent derivations in the setting of general metric spaces, especially as some cases are known.  For instance, both the doubling condition and the Poincar\'e inequality persist under measured pointed Gromov-Hausdorff limits \cite{Cheeger}, a process which generalises the previous blow-up procedure on $\R^n$.  Cheeger's Rademacher theorem then applies to the limiting metric space and a standard argument ensures that the induced differentials are derivations \cite{Weaver}, \cite{Gong_diffstruct}, \cite{Schioppa}.

For the general case of metric spaces with an MDS, the main challenge would be to replace smooth functions in the above proof with a suitable class of Lipschitz functions whose measurable differentials are invariant under the ``zooming out'' process of weak-star limits. (We daren't pursue this here.)

\section{Lip-lip conditions on Euclidean spaces} \label{sect_eucl}

We begin with subsets of $\R^n$ and $n$-dimensional Lebesgue measure, denoted by $\H^n$.  The following result is folklore, but we include a proof for convenience.

\begin{lemma} \label{lemma_euclLiplip}
If $A \subseteq \R^n$ is Lebesgue measurable with $\H^n(A) > 0$, then the metric measure space $(A,|\cdot|,\H^n)$ satisfies a Lip-lip condition with constant $M = 1$.
\end{lemma}
 
\begin{proof}
Let $f \in \Lip(A)$ be arbitrary.  If $F \in \Lip(\R^n)$ satisfies $F|_A = f|_A$, then
$$
\nabla f(x) \;:=\; \nabla F(x)
$$
is well-defined. 
Recalling that partial differential operators on $\R^n$ are derivations with respect to $\H^n$ \cite[Thm 37]{Weaver}, the locality property (Lemma \ref{lemma_locality}) implies that it is also independent of $F$, the choice of extension.  It is clear that
$$
\Lip[f|_A] \;\leq\;
\Lip[F|_A] \;\leq\;
|\nabla F| \;=\;
|\nabla f|
$$
holds a.e.\ on $A$.  Now fix $\e > 0$ and a Lebesgue point $x \in A$, and choose scales $(r_j)_{j=1}^\infty$ in $\R^+$ with $r_i \searrow 0$ so that
$$
\lim_{j \to \infty} 
L(f|_A;x,r_j) 
\;\leq\;
\lip[f|_A](x) \;+\; \e.
$$
Let ${\bf w}$ be a unit vector parallel to $\nabla F(x)$.  By the Lebesgue differentiation theorem, there exist $({\bf w}_j)_{j=1}^\infty \subset \bar{B}(0,1)$ so that $|{\bf w} - {\bf w}_j| \leq \e$ and $r_j{\bf w}_j \in A$.  Since $\nabla F(x)$ attains the maximal directional derivative of $F$ at $x$, we conclude that
\begin{eqnarray*}
|\nabla F(x)| &=&
\liminf_{j \to \infty} \frac{|F(x + r_j {\bf w}) - F(x)|}{r_j} \\ &\leq&
\liminf_{j \to \infty} \frac{|f(x + r_j {\bf w}_j) - f(x)|}{r_j} \,+\, \e \\ &\leq&
\liminf_{j \to \infty} \sup_{\bar{B}(x,r_j) \cap A} \frac{|f-f(x)|}{r_j} \,+\, \e  \;\leq\;
\lip[f|_A](x) + 2\e.
\end{eqnarray*}
The result follows from combining the above estimates and letting $\e \to 0$.
\end{proof}

It turns out that, up to measurable partitions, Lemma \ref{lemma_euclLiplip} also holds for general Radon measures $\nu$ on $\R^n$ that induce measurable differentiable structures.  Its proof uses Lemma \ref{lemma_liplinear} to reduce the class of admissible functions, so verifying the Lip-lip condition becomes a geometric problem.  More precisely, it suffices to study ``directions'' of differentiability at almost every point, and which of them attain the limits for $\lip[f](a)$ and $\Lip[f](a)$.

\begin{prop} \label{prop_euclconverse}
Let $\mu$ be a Radon measure on $\R^N$, let $S \subseteq \R^N$ be the support of $\mu$, and let $A \subseteq S$.  If $(A,|\cdot|,\mu)$ is a chart of differentiability for $S$, then there exist subsets $\{\mathcal{A}_n\}_{n=1}^\infty$ of $A$ so that $\mu(A \setminus \bigcup_n \mathcal{A}_n) = 0$ and for each $n \in \N$, we have
$$
\Lip[f](a) \;\leq\; n \, \lip[f](a)
$$
for all $f \in \Lip(\R^N)$ and for $\mu$-a.e.\ $a \in \mathcal{A}_n$.
\end{prop}

The proof splits into three parts.  At each point where Proposition \ref{prop_euclconverse} fails, ({\bf i}) the measure concentrates on slabs of arbitrarily small (relative) thickness.  As a result, ({\bf ii}) there must exist a tangent measure $\nu$ at that point that is supported on a hyperplane, so the rank of $\U(\R^N,\nu)$ must be at most $n-1$.  This leads to ({\bf iii}) a contradiction, since $\U(\R^N,\nu)$ must have rank $n$ by Corollary \ref{cor_tanderivrank}.

\begin{proof}
{\em Step} ({\bf 0}):\ {\em Setup}.\ 
From the chart condition on $A$ and Remarks \ref{rmk_lipderiv1side} and \ref{rmk_lipderiv}, there exists $K \geq 1$ so that, for all $f \in \Lip(\R^N)$ and for $\mu$-a.e.\ $a \in A$, we have
$$
\Lip[f](a) \;\leq\; K \, |{\bf D}f(a)|.
$$
Moreover, each component of $f \mapsto {\bf D}f$ is a derivation, so $\U(\R^N,\mu)$ has rank $N$.

Now suppose that there exists $h_1 \in \Lip(\R^N)$ so that 
$$
\Lip[h_1] \;>\; \lip[h_1]
$$
holds on a subset $A_1 \subseteq A$ with positive $\mu$-measure.  There are two cases:
\begin{itemize}
\item if $A_1$ satisfies a Lip-lip condition with $n = 2$, then the proof is complete;
\vspace{.025in}
\item otherwise, there exists $h_2 \in \Lip(\R^N)$ so that 
$$
\Lip[h_2] \;>\; 2 \lip[h_2]
$$
holds on a subset $A_2 \subseteq A_1$ with positive $\mu$-measure.
\end{itemize}
Iterating with $n = 1, 2, 3\ldots$ etc, either the Proposition holds true at some finite step, or there exist nested subsets $\{A_n\}_{n=1}^\infty$ of $A$ and $h_n \in \Lip(\R^N)$ so that
\begin{equation} \label{eq_contrary}
\Lip[h_n] \;>\; n \, \lip[h_n]
\end{equation}
holds $\mu$-a.e.\ on $A_n$, for all $n \in \N$.  By replacing $h_n$ with $L(h_n)^{-1}h_n$ as necessary, we further assume that $L(h_n) \leq 1$.  Now define the intersection
$$
A_\infty \;:=\; \bigcap_{n=1}^\infty A_n,
$$
so \eqref{eq_contrary} also holds $\mu$-a.e.\ on $A_\infty$, for each $n \in \N$.  

Now let $a \in A_\infty$ be a point of differentiability of $h_n$ for every $n \in \N$.  By the chart condition (\ref{defn_MDS}.B), this property applies to $\mu$-almost every point of $A_\infty$.

Let $\e \in \big(0, \frac{\Lip[h_n](a)}{n}\big)$ be given.

\vspace{.05in}
{\em Step} ({\bf i}):\ {\em Thin slabs}.\
For $n \in \N$, assume that $\Lip[h_n](a) > 0$. By Lemma \ref{lemma_liplinear} with $Y = A_n$ and $\xi = \id_{\R^N}$, inequality \eqref{eq_contrary} also holds for the function
$$
b \in \R^N \,\mapsto\, {\bf D}h_n(a) \cdot b \in \R
$$
at $a \in A_n$, so choose $(r_j)_{j=1}^\infty \subset \R^+$ with $r_j \searrow 0$ so that
\begin{eqnarray*}
\lim_{j \to \infty} \sup_{B(a,r_j)} \left| {\bf D}h_n(a) \cdot \frac{b-a}{r_j} \right|
&\leq&
\lip[h_n](a) \,+\, \e \\ &\leq&
\frac{1}{n}\Lip[h_n](a) \,+\, \e \;\leq\;
\frac{2}{n}\Lip[h_n](a).
\end{eqnarray*}
So for sufficiently large $j_n \in \N$, Remark (\ref{rmk_lipderiv1side}.B) implies the slab condition
\begin{eqnarray}
\left| {\bf D}h_n(a) \cdot (b-a) \right| &\leq& \notag
\sup_{B(a,r_j)} \left| {\bf D}h_n(a) \cdot \frac{b-a}{r_j} \right| r_j \;=\;
\frac{2( \Lip[h_n](a) + \e ) r_j}{n} \\ &\leq& \notag
\frac{3r_j}{n} \Lip[h_n](a) \;\leq\;
\frac{3K}{n} |{\bf D}h_n(a)| r_j \\
\label{eq_slab} \text{so }
\left| \frac{{\bf D}h_n(a)}{|{\bf D}h_n(a)|} \cdot (b-a) \right| &\leq&
\frac{3K}{M} r_j
\end{eqnarray}
holds for all $b \in B(a,r_j) \cap A_n$, whenever $j \geq j_n$; see Figure 1.

\begin{figure}[h]
\begin{pspicture}(-5,-1.25)(5,2)
\pscircle*[linecolor=gray](0,0){1.5}
\psarc*[linecolor=lightgray](0,0){1.5}{15}{165}
\psarc*[linecolor=lightgray](0,0){1.5}{195}{345}
\psline[linewidth=1pt]{->}(0,0)(0,1.8)\put(.1,1.8){$\frac{{\bf D}h_n(a)}{|{\bf D}h_n(a)|}$}
\pscircle*(0,0){.05}\put(-.3,-.1){$a$}
\psline[linewidth=.5pt]{<->}(3,-1.5)(3,0)\put(3.1,-.75){$r_j$}
\psline[linestyle=dashed](1.5,0)(2.9,0)
\psline[linestyle=dashed](0,-1.5)(2.9,-1.5)
\psline[linewidth=.5pt]{<->}(-3,-.375)(-3,0)\put(-3.9,-.3){$\frac{3K}{n}r_j$}
\psline[linestyle=dashed](-1.5,-.375)(-2.9,-.375)
\psline[linestyle=dashed](-1.5,0)(-2.9,0)
\end{pspicture}
\caption{The slab condition \eqref{eq_slab} for $B(a,r_j)$ in direction ${\bf D}h_n(a)$.}
\end{figure}
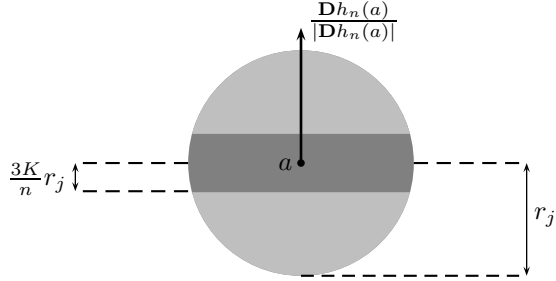

\vspace{.05in}
{\em Step} ({\bf ii}):\ {\em Tangent measures}.\
For the same point $a \in A_\infty$, the sequence of unit vectors
$\{ \frac{{\bf D}h_n(a)}{|{\bf D}h_n(a)|} ; n \in \N \}$
has a convergent subsequence which, with abuse of notation, we denote with the same symbols.

Call the limit ${\bf w}(a)$, and put $\rho_1 := r_{j_1}$ and by induction, for $M \in \N$ put 
$$
\rho_{n+1} \;:=\; \min\{\rho_n,r_{j_{n+1}}\} \, \text{ and } \,
c_n \;:=\; \frac{1}{\mu(B(a,\rho_n)}.
$$
For $\Omega := B(0,1)$, the sequence $\nu_n := c_n (T_{a,\rho_n})_\#\mu$ is norm-bounded, since
$$
\nu_n(B(0,1)) \;=\;
\frac{\mu\big(T_{a,\rho_n}^{-1}(B(0,1))\big)}{\mu(B(a,\rho_n))} \;=\;
\frac{\mu(B(a,\rho_n))}{\mu(B(a,\rho_n))} \;=\; 1,
$$
so up to a further subsequence, there is a Radon measure $\nu$ on $\Omega$ so that $\nu_n \wsto \nu$, hence $\nu \in \Tan(\mu,a)$.  Corollary \ref{cor_tanderivrank} and Lemma \ref{lemma_locality} therefore imply that $\U(\Omega,\nu)$ has rank $N$.

\begin{claim} \label{claim_hyperplane}
The measure $\nu|_\Omega$ is supported in the hyperplane 
$\Pi_a := a + {\bf w}(a)^\perp.$
\end{claim}

From this and Lemma \ref{lemma_locality} it would follow that $\U(\R^N,\nu|_\Omega)$ and $\U(\Pi_a,\nu|_\Omega)$ are isomorphic as modules.  Since $\Pi_a$ is isometric to $\R^{N-1}$, we would obtain a contradiction, since $\U(\R^N,\nu|_\Omega)$ would have rank at most $N-1$ and hence $\mu(A_\infty) = 0$.

\vspace{.05in}
{\em Step} ({\bf iii}):\ {\em Putting it together}.\
To prove Claim \ref{claim_hyperplane}, note that Condition \eqref{eq_slab} and the convergence $\frac{{\bf D}h_n(a)}{|{\bf D}h_n(a)|} \to {\bf w}(a)$ imply that, for sufficiently large $n \in \N$,
\begin{eqnarray*}
\left| {\bf w}(a) \cdot (b-a) \right| &\leq&
\left| \left( \frac{{\bf D}h_n(a)}{|{\bf D}h_n(a)|} - {\bf w}(a) \right) \cdot (b-a) \right| \,+\,
\left| \frac{{\bf D}h_n(a)}{|{\bf D}h_n(a)|} \cdot (b-a) \right| \\ &\leq&
\e |b-a| \,+\, \frac{3K}{M}\rho_n \;\leq\;
\frac{3K+1}{M} \rho_n
\end{eqnarray*}
holds for all $b \in B(a,\rho_n) \cap A_\infty$, where we used the explicit bound $\e < \frac{\Lip[h_n](a)}{n}$ from before (and where $L(h_n) \leq 1$).  In particular, the slabs
\begin{equation} \label{eq_slabmass}
\sigma_n \;:=\; \left\{
b \in \R^N \,;\,
|{\bf w}(a) \cdot (b-a)| \,\leq\, \frac{3K+1}{n} \rho_n
\right\}
\end{equation}
satisfy $\mu(B(a,\rho_n) \setminus \sigma_n) = 0$, for all $n \in \N$.

So for any ball $\bar{B}(b,R)$ in $\Omega \setminus {\bf w}(a)^\perp$, let $\varphi \in C_c(\Omega)$ be supported in $ \Omega \setminus {\bf w}(a)^\perp$ and satisfy $\varphi|_{B(b,R)} = 1$.  Since $\Omega \setminus {\bf w}(a)^\perp$ is open, there exists $n \in \N$ so that 
\begin{eqnarray*}
{\rm dist}\left( {\rm spt}(\varphi),\, {\bf w}(a)^\perp \right) &\leq&
\frac{3K+1}{n}, \\
\text{so } \, {\rm dist}\left( {\rm spt}(\varphi \circ T_{a,\rho_n}),\, \Pi_a \right) &\leq&
\frac{3K+1}{n}\rho_n
\end{eqnarray*}
and hence ${\rm spt}(\varphi \circ T_{a,\rho_n}) \cap \sigma_n = \emptyset$.  Following Definition \ref{defn_tanmeas}, we compute
\begin{eqnarray*}
\nu(B(b,R)) \;\leq\;
\int_\Omega \varphi \,d\nu &=&
\lim_{n \to \infty} c_n \int_{B(a,\rho_n)} (\varphi \circ T_{a,\rho_n}) \,d\mu \\ &\leq&
\lim_{n \to \infty} \frac{1}{\mu(B(a,\rho_n))} \int_{\sigma_n} (\varphi \circ T_{a,\rho_n}) \,d\mu \;=\; 0.
\end{eqnarray*}
Since $B(b,R)$ was arbitrary, Claim \ref{claim_hyperplane} follows.

As for the subsets in the Proposition, take $\mathcal{A}_n := A_n \setminus A_{n+1}$.
\end{proof}

\begin{rmk} \label{rmk_diffnotchart}
At this stage, a few observations about the proof are in order.
\begin{enumerate}
\item[(\ref{rmk_diffnotchart}.A)]
Proposition \ref{prop_euclconverse} is not a {\em quantitative} result, in that the constants $M \in \N$ for the Lip-lip conditions on $\mathcal{A}_M$ cannot be computed explicitly from $A$.

If we knew as in Lemma \ref{lemma_euclLiplip} that $d\mu \ll d\H_n$, then upon reaching some {\em finite} constant $M \in (3K,\infty)$, the slab condition \eqref{eq_slab} would already contradict the Lebesgue differentiation theorem.  Put otherwise, the lack of quantitativity in Proposition \ref{prop_euclconverse} is due to the lack of explicit information about measures on $\R^N$ that induce an MDS.

\vspace{.05in}
\item[(\ref{rmk_diffnotchart}.B)]
It is worthwhile to note that {\em the uniqueness of differentials is not needed in the proof of Proposition \ref{prop_euclconverse}}.
Instead, it suffices that $\mu$-almost every $a \in A$ is a point of differentiability, where there exists ${\bf v} \in \R^n$ satisfying Condition (\ref{defn_MDS}.A)  for $a$, as well as by applying Remarks \ref{rmk_lipderiv1side} and \ref{rmk_liplinear}.
\end{enumerate}
\end{rmk}

\section{Metric spaces:\ from differentiability to Lip-lip conditions} \label{sect_metric}

\subsection{Push, lift, {\em then} pull}

For Riemannian manifolds, tangent vectors allow pushforwards via diffeomorphisms.  A similar phenomenon holds true for measurable differentiable structures on metric spaces.

\begin{lemma} \label{lemma_MDSpushfwd}
Let $(X,d)$ be a complete metric space with a doubling measure $\mu$.  If 
$\{ (X_m,\xi^m) \}_{m=1}^\infty$
is an atlas for $(X,d,\mu)$, then each pair $(\xi^m(X_m), \id_{\R^{N_m}})$ is a chart of differentiability for $(\xi^m(X), |\cdot|, \xi^m_\#\mu)$.
\end{lemma}

As before, the proof proceeds in several steps.  First, ({\bf I}) we show Condition (\ref{thm_MDSchar}.B) holds for pushforward differentials, and then ({\bf II}) we verify Definition \ref{defn_MDS} directly.  To clarify, $\xi^m_\#\mu$ may be non-doubling, so Theorem \ref{thm_MDSchar} does not necessarily apply; we instead use smooth approximation of Lipschitz functions.

\begin{proof}
We work with one chart $X_m$ at a time.  To simplify notation, we therefore suppress the index $m$ and write $\xi = \xi^m$, $N = N_m$, and so on.

\vspace{.05in}

{\em Step} ({\bf I}):\ {\em verifying Condition (\ref{thm_MDSchar}.B)}.\
As indicated before in \S\ref{subsect_pushfwd}, by \cite[Lem 2.17]{Gong_rigidity} each component of $g \mapsto [\xi_\#{\bf D}]g$ is an element in $\U(\R^N,\xi_\#\mu)$.  So from the density of polynomials in $\Lip(\R^N)$ and Lemma \ref{lemma_pushfwd}, we conclude that the components of $\xi_\#{\bf D}$ form a basis of $\U(\R^N,\xi_\#\mu)$.

It remains to check the local-to-global inequality \eqref{eq_lipderiv}, so let $a \in \xi(X)$ and $\e > 0$ be given.  Choose $(r_j)_{j=1}^\infty \subset \R^+$ and $b_j \in B(a,r_j) \cap \xi(X)$, for each $j \in \N$, so that
$$
\Lip[g](a) -\e \;\leq\; 
\lim_{j \to \infty} L(g;\,a,r_j) \;=\;
\lim_{j \to \infty} \frac{|g(b_j) - g(a)|}{r_j}.
$$
Without loss, take preimages of $(b_j)$ that converge in $X$ to a preimage of $a$; to see this, letting $y_j  \in \xi^{-1}(\{b_j\})$ be arbitrary, the choice of coordinates \eqref{eq_distcoords} implies that $(y_j)_{j=1}^\infty$ is a bounded set, so by compactness there exists a convergent subsequence $(y_{j_k})_{k=1}^\infty$ with limit $x \in X$.  Continuity of distance functions then implies that
$$
\xi(x) \;=\;
\lim_{k \to \infty} \xi(y_{j_k}) \;=\;
\lim_{k \to \infty} b_{j_k} \;=\; a.
$$
Put $R_j := \min(r_j, |y_j-x|)$.  We now proceed to estimate
$$
\frac{|g(b_j) - g(a)|}{r_j} \;\leq\;
\frac{|(g \circ \xi)(y_j) - (g \circ \xi)(x)|}{R_j} \;\leq\;
L(g \circ \xi;\, x, R_j)
$$
which, combined with the previous estimates, further implies
$$
\Lip[g](a) - \e \;\leq\;
\lim_{j \to \infty} \frac{|g(b_j) - g(a)|}{r_j} \;\leq\;
\limsup_{j \to \infty} L(g \circ \xi;\, x,R_j) \;\leq\;
\Lip[g \circ \xi](x).
$$
The opposite inequality also holds; indeed, since $y \to x$ in $X$ implies $b \to a$, it follows that each $g \in \Lip(\R^N)$ satisfies
\begin{eqnarray*}
\limsup_{y \to x} \frac{|g(\xi(y)) - g(\xi(x))|}{d(x,y)} &=&
\limsup_{y \to x} \frac{|g(\xi(y)) - g(\xi(x))|}{d(x,y)} \, \frac{|\xi(y)-\xi(x)|}{|\xi(y)-\xi(x)|} \\ &\leq&
L(\xi) \, \limsup_{b \to a} \frac{|g(b) - g(a)|}{|b-a|}
\end{eqnarray*}
and if $a = b$, then the LHS is zero.  Letting $\e \to 0$, for $\mu$-a.e.\ $x \in X$ we have
\begin{equation} \label{eq_pullback}
\Lip[g](a) \;\leq\; 
\Lip[g \circ \xi](x) \;\leq\;
\sqrt{N} \Lip[g](a).
\end{equation}
This and Lemma \ref{lemma_pushfwd} imply Condition (\ref{thm_MDSchar}.B) on $\xi(X)$, with constant $\sqrt{N}K$.

\vspace{.05in}

{\em Step} ({\bf II}):\ {\em verifying Definition \ref{defn_MDS}}.\
For smooth $h \in \Lip(\R^N)$, the gradient $\nabla h$ is defined on all of $\R^N$ and satisfies Definition \ref{defn_MDS}.  Moreover, from the Chain Rule (Lemma \ref{lemma_chainrule}) and the pushforward formula (Lemma \ref{lemma_pushfwd}) it further follows that
$$
[\xi_\#{\bf D}]h(a) \;=\;
\nabla h(a) \cdot [\xi_\#{\bf D}]\id_{\R^N}(a) \;=\;
\nabla h(a) \cdot {\bf D}\xi(x) \;=\;
\nabla h(a)
$$
holds for $\xi_\#\mu$-a.e.\ $a \in \xi(X)$.

So for non-smooth $g \in \Lip(\R^N)$, let $t > 0$ and consider smooth, symmetric mollifiers $\eta_t : \R^N \to [0,\infty)$, supported on $\bar{B}(0,t)$, and put $h_t := g * \eta_t$.  Clearly $(h_t)_{t > 0}$ is uniformly Lipschitz and converges locally uniformly to $g$, so 
$$
[\xi_\#{\bf D}]h_t \wsto [\xi_\#{\bf D}]g \,\text{ in }\, L^\infty(\R^N,\xi_\#\mu)
$$
follows from Definition \ref{defn_derivation} and Lemma \ref{lemma_pushfwd}.  Fixing $p \in (1,\infty)$ and applying Mazur's lemma, the reflexivity of $L^p(\R^N,\xi_\#\mu)$, and the inclusion of spaces
$$
L^\infty_{\rm loc}(\R^N,\xi_\#\mu) \;\subset\; L^p_{\rm loc}(\R^N,\xi_\#\mu),
$$
then, up to taking convex combinations and subsequences of $(h_t)_{t > 0}$, we have
$$
[\xi_\#{\bf D}]h_t(a) \to [\xi_\#{\bf D}]g(a)
$$
pointwise for $\xi_\#\mu$-a.e.\ $a \in \xi(X)$.  Using \eqref{eq_pullback} at $\xi_\#\mu$-density points $a$, we have
\begin{equation*}
\begin{split}
& 
\Lip\left( g - [\xi_\#{\bf D}]g(a) \cdot \id_{\R^N} \right)(a) \\
& \hspace{1.in} \;\leq\;
\Lip( g - h_t )(a) \,+\,
\Lip( h_t - [\xi_\#{\bf D}]h_t(a) \cdot \id_{\R^N} )(a) \\ 
& \hspace{2.12in}
\,+\, \Lip( [\xi_\#{\bf D}](h_t - g)(a) \cdot \id_{\R^N} )(a) \\ 
& \hspace{1.in} \;\leq\;
\sqrt{N}\,|[\xi_\#{\bf D}](g-h_t)(a)| \,+\, 0 \,+\, N\,|[\xi_\#{\bf D}](g-h_t)(a)|
\end{split}
\end{equation*}
so the RHS vanishes as $t \to 0$.  The lemma follows.
\end{proof}

Proceeding with the analogy of Riemannian manifolds, recall that differential forms have natural pullbacks under smooth mappings.
With this in mind and the identity $\Lip[f] = |\nabla f|$ on $\R^n$, it is worth inquiring whether the Lip-lip condition is also preserved under pullback, in some reasonable sense.

For chart coordinates $\xi^m: X \to \R^{N_m}$ it is easy to show, from first principles, that pointwise Lipschitz constants in the target space $\R^{N_m}$ majorise those in $X$.  The converse is less clear.  To overcome this, we ``lift'' the coordinates to higher dimensions, so that the new geometry becomes more compatible with that of the source.  In particular, {\em quotients} of pointwise Lipschitz constants on the new target will be comparable to those on $X$.

\begin{lemma} \label{lemma_pullback}
Let $(X,d)$ be a complete metric space, let $\mu$ be a Radon measure on $X$, fix a finite set $\{x_i\}_{i=1}^N$ in $X$, and put
$$
\xi(y) \;:=\; ( \xi_i(y) )_{i=1}^N, \text{ where } \xi_i(y) \;:=\; d(x_i,y).
$$
Then for all $h \in \Lip(\R^{N+1})$ and $\mu$-a.e.\ $x \in X$, the inequality
$$
\frac{ \Lip[h \circ \zeta](x) }{ \lip[h \circ \zeta](x) } \;\leq\;
\sqrt{N+1} \, \frac{ \Lip[h](\zeta(x)) }{ \lip[h](\zeta(x)) } 
$$
holds, where $\zeta(y) := (\xi(y),d(x,y))$.
\end{lemma}

To fix notation, open cubes in $\R^n$, centered at $a = (a_1,\cdots,a_n)$, with edge length $\rho > 0$, and with faces orthogonal to the coordinate axes are denoted by
$$
Q_n(a,\rho) \;:=\; 
\left(a_1-\frac{\rho}{2},a_1+\frac{\rho}{2}\right) \times \cdots \times
\left(a_n-\frac{\rho}{2},a_n+\frac{\rho}{2}\right).
$$

\begin{proof}
As $\mu$ is Radon, assume $\mu(X) < \infty$. Fix a $\mu$-density point $x \in X$ and put
$$
g(y) \;:=\; d(x,y).
$$  
Since the components of $\xi$ are distance functions, for sufficiently small $\rho > 0$ the preimage of $Q_N(\xi(x),\rho)$ is a finite intersection of open annuli in $X$, each of thickness $\rho$, and hence a bounded open neighborhood of $\xi^{-1}(\{ \xi(x) \})$. It follows that
$$
\xi(B(x,\rho)) \;\subseteq\; Q_N(\xi(x),\rho)
$$
holds for all $x \in X$ and all $0 < \rho \leq \min\{ \xi_i(x) : 1 \leq i \leq N \}$ and hence
\begin{eqnarray*}
\zeta(B(x,\rho)) \;=\;
\xi(B(x,\rho)) \times g(B(x,\rho)) &\subseteq&
Q_N(\xi(x),\rho) \times [g(x)-\rho,g(x)+\rho] \\ &=&
Q_{N+1}(\zeta(x),\rho)
\end{eqnarray*}
On the other hand, for points $y \in \zeta^{-1}(Q_{N+1}(\zeta(x),\rho))$, it is clear that
$$
d(x,y) \;=\;
|g(y) - g(x)| \;\leq\;
|\zeta(y) - \zeta(x)| \;\leq\;
\rho
$$
which further implies the set inclusion
$$
\zeta^{-1}\big( Q_{N+1}(\zeta(x),\rho) \big) \;\subseteq\; B(x,\rho).
$$
It follows that the previous set inclusions reduce to an identity
\begin{equation} \label{eq_ballcube}
\zeta\big(B(x,\rho)\big) \;=\; \zeta(X) \cap Q_{N+1}(\zeta(x),\rho)
\end{equation}
and moreover, that $\zeta(x)$ is a $\zeta_\#\mu$-density point in $\R^{N+1}$.

By identifying $\R^N \times \{0\}$ with $\R^N$ and letting $\pi : \R^{N+1} \to \R^N$ denote orthogonal projection onto the first $N$ coordinates in $\R^{N+1}$, we see that
$$
\pi_\#(\zeta_\#\mu)(A) \;=\;
\zeta_\#\mu(A \times \R) \;=\;
\mu(\zeta^{-1}(A \times \R)) \;=\;
\mu(\xi^{-1}(A)) \;=\; \xi_\#\mu(A)
$$
holds, for all Borel sets $A \subset \R^N$, so by Borel regularity, we conclude that
\begin{equation} \label{eq_measpushdown}
\pi_\#(\zeta_\#\mu) \;=\; \xi_\#\mu.
\end{equation}
Letting $h \in \Lip(\R^{N+1})$ and $\e > 0$ be arbitrary, choose radii $(R_j)_{j=1}^\infty \searrow 0$ so that
\begin{eqnarray*}
\lip[h](\zeta(x)) \,+\, \e &\geq&
\liminf_{j \to \infty} \sup_{Q_{N+1}(\zeta(x),R_j)} \frac{|h - h(\zeta(x))|}{R_j} \\ &\geq&
\liminf_{j \to \infty} \sup_{\zeta(B(x,R_j))} \frac{|h - h(\zeta(x))|}{R_j} \\ &\geq&
\liminf_{j \to \infty} \sup_{B(x,R_j)} \frac{|h \circ \zeta - (h \circ \zeta)(x)|}{R_j} \;=\; 
\lip[h \circ \zeta](x)
\end{eqnarray*}
and choose radii $(r_j)_{j=1}^\infty \searrow 0$ and an index $j_0 \in \N$ so that, for all $j \geq j_0$, we have
\begin{eqnarray*}
\lip[h \circ \zeta](x) \,+\, \e &\geq&
\liminf_{j \to \infty} \sup_{B(x,r_j)} \frac{|h \circ \zeta - (h \circ \zeta)(x)|}{r_j} \\ &=&
\liminf_{j \to \infty} \sup_{\zeta(B(x,r_j))} \frac{|h - h(\zeta(x))|}{r_j} \\ &=&
\liminf_{j \to \infty} \sup_{\zeta(X) \cap Q_{N+1}(\zeta(x),r_j)} \frac{|h - h(\zeta(x))|}{r_j} \;\geq\;
\lip[h]( \zeta(x) ).
\end{eqnarray*}
Combining the above estimates and letting $\e \to 0$, it follows that
$$
\lip[h \circ \zeta](x) \;=\; \lip[h]( \zeta(x) )
$$
holds for all $h \in \Lip(\R^{N+1})$; the inequality
$$
\Lip[h \circ \zeta](x) \;\leq\; 
L(\zeta) \Lip[h](\zeta(x)) \;\leq\;
\sqrt{N+1} \Lip[h](\zeta(x)),
$$
however, is straightforward.
\end{proof}

\subsection{Slicing the tangent measures}

The proof of our main result, Theorem \ref{thm_converse}, follows that of Proposition \ref{prop_euclconverse}.  Namely, the process of taking tangent measures on the new target $\R^{N_m+1}$ corresponds to a similar process on $\R^{N_m} \times \{0\}$.

\begin{proof}[Proof of Theorem \ref{thm_converse}]
We show (\ref{thm_converse}.A) $\Rightarrow$ (\ref{thm_converse}.B).
As $\mu$ is doubling and hence Radon, assume $\mu(X_m) < \infty$ for all $m \in \N$.  As usual, we suppress the index $m$, so $X = X_m$, $\xi = \xi^m$, etc.  Fix a $\mu$-density point $x \in X$ and put
$$
g(y) \;:=\; d(x,y).
$$  
Assume all the notation and background from the proof of Lemma \ref{lemma_pullback}, so in particular, we write $\zeta(y) := (\xi(y),g(y))$.  Points in $\R^{N+1}$ are denoted as pairs
$$
(a,s) \in \R^N \times \R.
$$
For every $f \in \Lip(X)$, Lemma \ref{lemma_pullback} applies to the auxiliary function
$$
h_f(a,s) \;:=\; {\bf D}f(x) \cdot a,
$$
so if $f$ is differentiable at $x$ with respect to $\xi$ in the sense of (\ref{defn_MDS}.A), then $h_f$ is differentiable at $\zeta(x)$ with differential
$$
{\bf v} \;:=\; [\zeta_\#{\bf D}]h_f(\zeta(x)) \;=\; \big( {\bf D}f(x), 0 \big)
$$
and with respect to the identity map on the subset
$$
\zeta(X) \;=\; \xi(X) \times g(X) \;\subseteq\; \R^N \times \R.
$$
Supposing that the Lip-lip condition fails on all of $X$ --- that is, for each $n \in \N$ there exist a subset $Y_n \subset X$ with $\mu(Y_n) > 0$ and $(f_n)_{n=1}^\infty$ in $\Lip(X)$ so that
$$
\Lip[{\bf D}f_n \cdot \xi] \;=\;
\Lip[f_n] \;>\; 
n \, \lip[f_n] \;=\;
n \, \lip[{\bf D}f_n \cdot \xi]
$$
holds $\mu$-a.e.\ on $Y_n$ --- then an analogous condition holds $\zeta_\#\mu$-a.e.\ on $\zeta(Y_n)$, i.e.
$$
\Lip[h_{f_n}] \;>\; n \, \lip[h_{f_n}].
$$
Without loss, suppose that the subsets $\{Y_n\}_{n=1}^\infty$ are nested under inclusion.

Verily, by (\ref{rmk_diffnotchart}.B) the same argument in Step ({\bf i}) of Proposition \ref{prop_euclconverse} applies, with dimension $N+1$, $A_n := \zeta(Y_n)$, and $h_n := h_{f_n}$, and with $\zeta_\#\mu$ in place of $\mu$.  Putting 
$$
(a,s) \;:=\; (\xi(x),0) \;=\; \zeta(x)
$$
and with the same abuse of notation for subsequences, there is a limit
$$
\frac{[\zeta_\#{\bf D}]h_n(a,0)}{|[\zeta_\#{\bf D}]h_n(a,0)|} \;\To\; {\bf w}(a) \in \R^{N+1} \, \text{ as } \, n \to \infty
$$
as well as thicknesses $\rho_n > 0$, constants $c_n$, and slabs 
$$
\sigma_n \;:=\; \left\{
(b,t) \in \R^N \times \R \,;\,
\left|{\bf w}(a) \cdot \big((b,t)-(a,0)\big)\right| \,\leq\, \frac{3K+1}{n} \rho_n
\right\}
$$
as initially given in \eqref{eq_slabmass} and where $K \geq 1$ comes from Remark \ref{rmk_lipderiv1side}.  As before in Step ({\bf ii}), there exists a weak-star limit of probability measures
$$
\nu_n \;:=\; c_n (T_{(a,0),\rho_n})_\#(\zeta_\#\mu) \;\wsto\; \nu
$$
where the tangent measure $\nu \in \Tan(\zeta_\#\mu, (a,0))$ satisfies $\nu(\R^{N+1} \setminus \sigma_n) = 0$ for all $n$.  By construction, moreover, we have that
$$
\frac{({\bf D}f_n(x),0)}{|{\bf D}f_n(x)|} \;=\; 
\frac{[\zeta_\#{\bf D}]h_n(a,0)}{|[\zeta_\#{\bf D}]h_n(a,0)|} \;\To\; 
{\bf w}(a) \;=:\;
(\hat{\bf w}(a), 0)
$$
for some $\hat{\bf w}(a) \in \R^N$.  This means that the coordinate hyperplane $\R^N \times \{0\}$ is orthogonal to the hyperplane $(a,0) + {\bf w}(a)^\perp$ in $\R^{N+1}$--- where the mass of $\nu$ is supported --- and the slabs $\sigma_M$ intersect it in lower-dimensional slabs of the form
$$
\hat\sigma_n \;:=\; \left\{
b \in \R^N \,;\, 
\left| \hat{\bf w}(a) \cdot (b-a) \right| \,\leq\, \frac{3K+1}{n} \rho_n
\right\}.
$$
The rigid motions of projection and translation are almost commutative:
$$
\pi \circ T_{(a,0),\rho_n} \;=\; T_{a,\rho_n} \circ \pi.
$$
From the above identity and \eqref{eq_measpushdown} the probability measures $\hat\nu_n := \pi_\#\nu_n$  are supported in $\xi(X) \subset \R^N$ and obey
\begin{eqnarray*}
\hat\nu_n \;=\;
\pi_\#\nu_n &=& 
c_n \pi_\#(T_{(a,0),\rho_n})_\#(\zeta_\#\mu) \\ &=&
c_n (T_{a,\rho_n})_\#(\pi_\#(\zeta_\#\mu)) \;=\;
c_n (T_{a,\rho_n})_\#(\xi_\#\mu).
\end{eqnarray*}
Take a convergent subsequence and call the limit $\hat\nu$; in particular, $\hat\nu \in \Tan(\xi_\#\mu,a)$.

With the sub-slabs $\{\hat\sigma_n\}_{n=1}^\infty$ in place of the $\{\sigma_n\}_{n=1}^\infty$, an analogous argument as in Step ({\bf iii}) of Proposition \ref{prop_euclconverse} shows that $\hat\nu$ must be supported in the hyperplane $a + \hat{\bf w}(a)^\perp$ in $\R^N$, so the rank of $\U(\R^N,\hat\nu)$ is at most $N-1$.

This, of course, contradicts Theorem \ref{thm_tanderiv} and Lemma \ref{lemma_MDSpushfwd}.  As a result, the set $\bigcap_{n=1}^\infty Y_n$ must have zero $\mu$-measure.  Subdividing the chart $X = X_m$ into the subsets $Z_n := Y_n \setminus Y_{n+1}$, putting $M_n := n\sqrt{N+1}$ and invoking Lemma \ref{lemma_pullback}, then as in the final step of the proof of Proposition \ref{prop_euclconverse}, we see that with lifted coordinates $\zeta = \zeta^m$
$$
\frac{\Lip[f](x)}{\lip[f](x)} \;\leq\;
\frac{ \sqrt{N_m+1} \Lip[h_f](\zeta^m(x))}{ \lip[h_f](\zeta^m(x)) } \;\leq\; n\sqrt{N+1} \;=\; M_m
$$
holds for $\mu$-a.e.\ $x \in Z_n$, where now $N := \sup_n N_n < \infty$.
\end{proof}

Similarly to Proposition \ref{prop_euclconverse}, observe that Theorem \ref{thm_converse} is not a quantitative statement.  In the case of an $N$-dimensional MDS on $(X,d,\mu)$, for $N \leq 2$, the main result from \cite{Gong_rigidity} asserts that pushforwards of doubling measures $\mu$ on $X$ enjoy absolute continuity with Lebesgue measure, that is:\
$$
\xi^m_\#\mu \;\ll\; \H^{N_m}.
$$
Lemma \ref{lemma_euclLiplip} then applies, so each $\xi^m(X_m)$ satisfies a Lip-lip condition with constant $M = 1$.  Corollary \ref{cor_converse} follows with the same (remaining) argument as Theorem \ref{thm_converse}.  

\section{Appendix: Differentiability with minimal hypotheses} \label{sect_bate}

Inspired by Bate's result \cite{Bate} we now present an independent proof of Theorem \ref{thm_bate}, as well as a new characterisation of measurable differentiable structures on general metric spaces, without any additional assumptions on the underlying Radon measure.  The latter result is stated below, and generalises Theorem \ref{thm_MDSchar}.

\begin{prop} \label{prop_MDSderivs}
Let $X = (X,d)$ be a metric space and let $\mu$ be a Radon measure on $X$.  Then $(X,d,\mu)$ supports a non-degenerate measurable differentiable structure if and only if both of the following conditions hold:\
\begin{enumerate}
\vspace{.05in}
\item[(\ref{prop_MDSderivs}.A)] 
the measure $\mu$ is pointwise doubling, i.e.\ for $\mu$-a.e.\ $x \in X$, it holds that
$$
\limsup_{r \to 0} \frac{\mu(B(x,2r))}{\mu(B(x,r))} \;<\; \infty
$$
\item[(\ref{prop_MDSderivs}.B)]
there exist a collection of $\mu$-measurable subsets $\{X_l\}_{l=1}^\infty$ of $X$ and sequences $\{N_l\}_{l=1}^\infty \subseteq \N$ and $\{K_l\}_{l=1}^\infty \subseteq [1,\infty)$ so that for each $m \in \N$, there is a basis ${\bf d}^l = (\d_i^l)_{i=1}^{N_l}$ in $\U(X_m,\mu)$ so that for all $f \in \Lip(X)$, the local-to-global inequality \eqref{eq_lipderiv} holds $\mu$-a.e.\ on $X_l$.
\end{enumerate}
\end{prop}

%
%

We begin by noting that the Vitali covering theorem (and hence the Lebesgue differentiation theorem) also holds for pointwise doubling measures $\mu$ on $X$.  Indeed, similarly as in \cite[p.\ 45]{Bate} one subdivides $X$ into countably many subsets
\begin{equation} \label{eq_doublingpieces}
X_{m,n} \;:=\;
\{
x \in X \,;\,
\mu(B(x,2r)) \;\leq\; 2^n \mu(B(x,r)) \text{ for all } r \in (0,2^{-m})
\}
\end{equation}
for $n,m \in \N$; indeed, by hypothesis almost every $x \in X$ satisfies
$$
\limsup_{r \to 0} \frac{\mu(B(x,2r))}{\mu(B(x,r))} \;\leq\; 2^{n-1}
$$
for some $n \in \N$, so $x \in X_{n,m}$ holds for sufficiently large $m \in \N$.  The same proofs for doubling measures therefore apply to $X_{n,m}$ and hence to $X$.


In a similar spirit, the next result asserts that metric spaces supporting pointwise doubling measures are countable unions of subsets, each of which is a doubling metric space, in the sense of (\ref{rmk_doubling}.B).  

\begin{lemma} \label{lemma_ptdoubling}
If $\mu$ is Radon and pointwise doubling on a metric space $X$, then there is a collection of $\mu$-measurable subsets $\{Z_{n,l}\}_{n,l=1}^\infty$ of $X$ with $\mu(X \setminus \bigcup_{n,l} Z_{n,l}) = 0$ and where, for each $(n,l) \in \N \times \N$,
\begin{enumerate}
\item[(\ref{lemma_ptdoubling}.A)] the subset $Z_{n,l}$ is $N$-doubling in the sense of (\ref{rmk_doubling}.A), for some $N \in \N$;
\item[(\ref{lemma_ptdoubling}.B)] the restricted measure $\mu\lfloor_{Z_{N,L}}$ satisfies the doubling condition \eqref{eq_doubling} for all radii $r \in (0,2^{-l})$ with constant $\kappa = 2^n$.
\end{enumerate}
\end{lemma}

In particular, (\ref{lemma_ptdoubling}.B) is precisely \cite[Lem 8.3]{Bate}, so we prove only (\ref{lemma_ptdoubling}.A).

\begin{proof}
Let $X_{n,m}$ be the subsets defined in \eqref{eq_doublingpieces}; without loss, assume that each has positive $\mu$-measure.  By the Lebesgue differentiation theorem for $f = \chi_{X_{n,m}}$, almost every $x \in X_{n,m}$ therefore satisfies
$$
\frac{\mu(B(x,r) \cap X_{n,m})}{\mu(B(x,r))} \;\geq\; \frac{1}{2}
$$
for small enough $r = r(x) > 0$.  So by subdividing each $X_{n,m}$ into further subsets
$$
\{ x \in X_{n,m} \setminus X_{n+1,m} \,;\, 2^{-l-1} \,\leq\, r(x) \,<\, 2^{-l}\}
$$
and re-indexing as necessary, the lemma follows.
\end{proof}

To prove Proposition \ref{prop_MDSderivs}, we will use the necessity of the pointwise doubling condition, which has already been established by Bate and Speight \cite[Cor 2.6]{Bate:Speight}. 

\begin{lemma}[Bate-Speight] \label{lemma_MDSptdoubling}
Let $X = (X,d)$ be a metric space with a locally finite Borel measure $\mu$.  If $(X,d,\mu)$ supports a nondegenerate measurable differentiable structure, then $\mu$ must be pointwise doubling.
\end{lemma}

\begin{proof}[Proof of Proposition \ref{prop_MDSderivs}]
($\Leftarrow$) Let $X_l$ be one of the subsets from Condition (\ref{prop_MDSderivs}.B) with positive $\mu$-measure.  Since $\mu$ is pointwise doubling on $X$, the Lebesgue differentiation theorem implies that $\mu\lfloor_{X_l}$ is also pointwise doubling.  
So by taking $\e$-nets $Y_{m,l}^\e$ of each $Y_{m,l}$, for all $\e > 0$, as well as ``piecewise-distance'' approximations \cite[Defn 4.1]{Gong_diffstruct} of each $f \in \Lip(X)$, defined as
\begin{equation} \label{eq_piecewise}
f_\e(x) \;:=\; \inf\{ f(y) + L(f) \, d(x,y) \,;\, y \in Y^\e_{m,l} \},
\end{equation}
the remainder of the proof follows that of Theorem \ref{thm_MDSchar} in \cite[p.\ 21-23]{Gong_diffstruct}.

\vspace{.05in}

($\Rightarrow$) Assume that $(X,d,\mu)$ supports an MDS; without loss, $X$ is a single chart with coordinates $\xi : X \to \R^N$.  Applying Lemma \ref{lemma_MDSptdoubling}, let $\{Y_{n,l}\}_{n,l=1}^\infty$ be the collection of subsets from Lemma \ref{lemma_ptdoubling}.  Assume that each $Y_{n,l}$ has positive $\mu$-measure, so each $(Y_{n,l},d,\mu)$ also supports an MDS.  By Remark \ref{rmk_lipderiv}, inequality \eqref{eq_lipderiv} holds $\mu$-a.e.\ on $Y_{n,l}$ with the same differential map $f \mapsto {\bf D}f$.

To show that the components of $f \mapsto {\bf D}f$ are derivations, the same argument as in the proof of \cite[Thm 1.6]{Gong_diffstruct} uses only functional analytic techniques, so it runs as before with one modification:\ the doubling condition was used to invoke \cite[Thm 9]{Franchi:Hajlasz:Koskela}, which asserts that for $p \geq 1$ the Haj{\l}asz-Sobolev space $M^{1,p}(X,\mu)$ is contained in $H^{1,p}(X,\mu)$, the completion of the linear space of functions
$$
\tilde{H}^{1,p}(X,\mu) \;:=\;
\{ f \in {\Lip}_{\rm loc}(X) \cap L^p(X,\mu) \,;\, |{\bf D}f| \in L^p(X,\mu) \}
$$
with respect to the norm 
$$
\|f\|_{H^{1,p}(X,\mu)} \;:=\; \|f\|_{L^p(X,\mu)} \,+\, \| |{\bf D}f| \|_{L^p(X,\mu)}.
$$
However, a close reading of that proof shows that the  doubling condition for $\mu$ is used only in two cases:
\begin{itemize}
\item in \cite[p.\ 1908]{Franchi:Hajlasz:Koskela} the doubling space property (\ref{rmk_doubling}.B) of $X$ is used to obtain coverings of $X$ by balls of small uniform radius $\e > 0$ and uniformly bounded overlap, which in turn gives rise to approximations via Lipschitz partitions of unity.  Lemma \ref{lemma_countdoubling} can therefore be used for each $Z_{n,l} \cap Y_m$;
\vspace{.05in}
\item the estimates in \cite[p.\ 1916-1918]{Franchi:Hajlasz:Koskela} only use the covering balls of fixed radius $\e > 0$, as before, so Lemma \ref{lemma_ptdoubling} applies instead: $\mu\lfloor_{Y_{n,l}}$ is doubling for sufficiently small radii, so it suffices to take $\e \in (0,2^{-l})$.
\end{itemize}
This settles the remaining implication.
\end{proof}

\begin{rmk}
For doubling measures, Lemma \ref{lemma_doublingrank} is used to fix a dimension bound for measurable differentiable structures on the underlying space.  This is not always possible for the non-doubling case, however, and therefore not needed (or used) for the proof of Proposition \ref{prop_MDSderivs}.

As an example, fix an infinite-dimensional Hilbert space $H$ with an orthonormal basis $({\bf e}_i)_{i=1}^\infty$ and write ${\bf 0} \in H$ for the zero element.  For the subsets
$$
X_m \;:=\; [0,1] \times \R^{m-1} \times \{ {\bf 0} \}
$$
the union $X := \bigcup_{m=1}^\infty (X_m + m{\bf e}_1)$ supports an MDS with $N = \infty$ and with respect to the sum of $m$-dimensional Hausdorff measures
$$
d\mu \;:=\; \sum_{m=1}^\infty \chi_{X_m}\,d\mathcal{H}^m
$$
which is pointwise doubling in the sense of (\ref{prop_MDSderivs}.A) but fails \eqref{eq_doubling}.  It is moreover clear that each $\U(X_m,\mu)$ has rank-$m$.
\end{rmk}

Before proceeding to Theorem \ref{thm_bate}, we will need a more general version of (\ref{rmk_doubling}.C).  An alternate argument can be found in \cite[Cor 6.28]{Schioppa}.

\begin{cor} \label{cor_distcoords}
If $(X,d,\mu)$ supports a measurable differentiable structure, then there is an atlas $\{(X_m,\xi^m)\}_{m=1}^\infty$ on $X$ so that each chart coordinate $\xi^m : X \to \R^{N_m}$ consists of distance functions.
\end{cor}

\begin{proof}
Assume all the notation from the proof of Proposition \ref{prop_MDSderivs}.

For $(\e_j)_{j=1}^\infty$ in $\R^+$ with $\e_j \to 0$, the functions from \eqref{eq_piecewise} satisfy $f_{\e_j} \wsto f$ in $\Lip_b(X)$, so by weak continuity we obtain
$\d_i^mf_{\e_j} \wsto \d_i^mf$
in $L^\infty(Y_{m,l},\mu)$ for each $i$, $m$.  In particular, enumerating $Y^{\e_j}_{m,l} = \{y_n^j\}_{n=1}^\infty$ and putting
$$
E_n^j \,:=\, \{ x \in X \,;\, f_\e(x) = f(y_n^j) + L(f) \, d_{y_n^j}(x) \}
$$
the function $\d_i^mf_{\e_j}$ then takes the form
$$
\d_i^mf_{\e_j} \,=\, 
L(f) \sum_{n=1}^\infty \chi_{E_n^j} \d_i^md_{y_n^j}.
$$
Similarly as Step ({\bf II}) in the proof of Lemma \ref{lemma_MDSpushfwd}, a Mazur's lemma argument shows that $\d_i^mf$ is an $L^\infty(X,\mu)$-linear combination of the functions $\{\d_i^md_{y_n^j}\}_{j,n=1}^\infty$.  The rest of the proof follows with the same linear algebra argument as in the proof of \cite[Lem 2.12]{Gong_diffstruct} as well as the ``change of variables'' trick in the proof of \cite[Thm 3.2]{Gong_diffstruct}.
\end{proof}

We conclude with an outline of the modifications to the proof of Theorem \ref{thm_converse}, so that Theorem \ref{thm_bate} follows:

\begin{proof}[Sketch of Proof for Theorem \ref{thm_bate}]
Assume $(X,d,\mu)$ has a nondegenerate MDS, so $\mu$ is pointwise doubling by Lemma \ref{lemma_MDSptdoubling}.
Lemma \ref{lemma_pullback} applies to this setting, since Corollary \ref{cor_distcoords} implies the existence of an atlas on $X$ with distance functions as coordinates $\xi^m$ on each chart $X_m$.

The proof of Lemma \ref{lemma_MDSpushfwd} also relies on distance functions as coordinates to ensure that $\xi^m_\#\mu$ is locally finite, so Corollary \ref{cor_distcoords} also applies here in place of the doubling condition.  The only other use of doubling comes from Lemma \ref{lemma_pushfwd}, which uses the fact that doubling measures satisfy Vitali's Covering Theorem and are used to build Lipschitz partitions of unity, as from \cite{Franchi:Hajlasz:Koskela}. The first property follows from the use of the subsets $X_{n,m}$ in \eqref{eq_doublingpieces}; for the second, the same observation as for ($\Rightarrow$) in Proposition \ref{prop_MDSderivs} works.

The remainder of the proof of Theorem \ref{thm_converse} only uses differentiability, pointwise Lipschitz constants, and Theorem \ref{thm_tanderiv} and Proposition \ref{prop_euclconverse}, which only require the underlying measure to be Radon, so the argument runs as before.
\end{proof}

\bibliographystyle{alpha}
\bibliography{lip}
\end{document}